\documentclass[final]{siamart0516}
\usepackage{amsmath,amsfonts, bbm}
\usepackage{tikz}

\DeclareMathOperator*{\esslim}{ess\;lim}
\DeclareMathOperator*{\esssup}{ess\;sup}
\DeclareMathOperator*{\essinf}{ess\;inf}

\def\a{\alpha}
\def\b{\beta}
\def\e{\epsilon}
\def\del{\partial}
\def\eps{\varepsilon}
\def\ol{\overline}
\def\wt{\widetilde}
\def\wh{\widehat}
\def\<{\big\langle}
\def\>{\big\rangle}
\def\mc{\mathcal}

\def\NN{\mathbb{N}}
\def\DD{\mathbb{D}}
\def\RR{\mathbb{R}}
\def\11{\mathbbm{1}}
\def\com#1{\quad\textrm{#1}\quad}
\def\TS#1{{\textstyle{#1}}}
\def\wint{\star\!\!\int}
\def\qn#1{\big|\!\big|\!\big|\,#1\,\big|\!\big|\!\big|}

\title{Weak* Solutions II: The Vacuum in Lagrangian Gas Dynamics \\{\,}\vspace{8pt}
 \tiny (\rm In: SIAM Journal on Mathematical Analysis (2017), 49(3), 1810-1843.) }

\author{Alexey Miroshnikov\thanks {Department of Mathematics, University of
California, Los Angeles, amiroshn@gmail.com} \and
 Robin Young\thanks{Department of Mathematics and Statistics, University of
Massachusetts, young@math.umass.edu}}

\headers{Weak* Solutions II: The Vacuum}
{Alexey Miroshnikov and Robin Young}

\begin{document}

\maketitle

\begin{abstract}
We develop a framework in which to make sense of solutions containing
the vacuum in Lagrangian gas dynamics.  At and near vacuum, the
specific volume becomes infinite and enclosed vacuums are represented
by Dirac masses, so they cannot be treated in the usual weak sense.
However, the weak* solutions recently introduced by the authors can be
extended to include solutions containing vacuums.  We present a
definition of these \emph{natural vacuum solutions} and provide
explicit examples which demonstrate some of their features.  Our
examples are isentropic for clarity, and we briefly discuss the
extension to the full $3\times3$ system of gas dynamics.  We also
extend our methods to one-dimensional dynamic elasticity to show that
fractures cannot form in an entropy solution.
\end{abstract}



\section{Introduction}

The oldest and most fundamental system of hyperbolic conservation laws
is that of isentropic gas dynamics, which are the simplest analog of
Newton's Law for a continuous medium.  The equations can be expressed
either in an Eulerian spatial frame, or in a Lagrangian or co-moving
material frame.  In the Lagrangian frame, the equations are
\begin{equation}
  \del_tv - \del_xu = 0, \qquad
  \del_tu + \del_xp = 0,
\label{psyst}
\end{equation}
where $x$ is the material variable, $v$ is the specific volume, and
$u$ and $p$ are the fluid velocity and pressure, respectively.  The
system is closed by specifying a constitutive relation $p = P(v)$, a
monotone decreasing function which is integrable as $v\to\infty$.
Alternatively, in an Eulerian frame, the equations are
\[
  \del_t\rho + \del_y(\rho\,u) = 0, \qquad
  \del_t(\rho\,u) + \del_y(\rho\,u^2 + p) = 0,
\]
representing conservation of mass and momentum, respectively, where
$y$ is the spatial variable, and $\rho = 1/v$ is the density.

The main effect of nonlinearity in a hyperbolic system is the presence
of shock waves, across which the pressure and velocity are
discontinuous, and the equations cannot be satisfied in the classical
sense.  This problem is usually solved by the use of weak solutions,
which are defined by multiplying by test functions and integrating by
parts.  There is now a mature and largely complete theory of $BV$ weak
solutions of systems of conservation laws, provided the data is
appropriately small~\cite{Bressan}.

Another effect of nonlinearity is the presence of a vacuum, which
corresponds to $v=\infty$ or $\rho=0$.  The vacuum presents different
difficulties depending on the frame: in an Eulerian frame, the
equations degenerate and the velocity $u$ is underdetermined, while in
a Lagrangian frame the vacuum is formally described using a Dirac
mass, so the class of weak solutions is not large enough.  The goal of
this paper is to rigorously justify the use of Dirac masses and thus
present a satisfactory notion of solution which includes vacuums in a
Lagrangian frame.

In the recent paper~\cite{MY16}, the authors introduced the notion of
weak* solution, which we believe holds several advantages over weak
solutions.  Our approach is natural and general, and allows us to view
the system as an evolutionary ODE in Banach space, which in turn
confers some regularity.  In addition, the ``multiplication by test
function and integration by parts'' step is treated abstractly rather
than explicitly, leading to cleaner calculations.  Our approach is
also general enough to handle certain extensions, including the
treatment of vacuums as Dirac masses.  In ~\cite{MY16}, we also proved
that $BV$ weak solutions are weak* solutions and \emph{vice versa},
which implies that the well-known uniqueness and regularity results
for $BV$ solutions apply unchanged to weak* solutions.

To define a weak* solution of an abstract system of conservation laws,
\[
  \del_tU + \del_xF(U) = 0,
\]
we begin with a normed vector space $X$ of \emph{spatial} test
functions, and regard the solution $U(t)$ as a function taking values
in the dual space $X^*$ of $X$.  For $BV$ solutions, we take $X =
C_0(\Omega)^n$ with $\Omega \subset \RR $, so that $X^* = M(\Omega)^n$, the space of Radon
measures.  Then if $U\in BV_{loc}^n$, so is $F(U)$, and so the
distributional derivative $\DD_xF(U)\in X^*$.  We then declare $U$ to
be a weak* solution if is satisfies the Banach space ODE
\begin{equation}
  U' + \DD_xF(U) = 0 \com{in} X^*_{loc}=M_{loc}(\RR)^n,
\label{Bode}
\end{equation}
where $U'$ is the appropriate time derivative of $U(t)$.  This is the
Gelfand weak* derivative, or G-derivative, defined by using the
Gelfand weak* integral of functions $\phi:[0,T]\to X^*$.  The
associated spaces are labelled $W^{1,q}_{w*}(0,T;X^*)$.

In this paper, we extend the ideas of~\cite{MY16} to include the use
of Dirac masses in weak* solutions.  The key observation is that in
\cref{psyst}, although $v$ is no longer bounded, or even a function, the
flux vector $(u,p)$ remains $BV$, so that its spatial derivative is a
measure, so lives in $X^*$, and the ODE \cref{Bode} makes sense.
Instead of treating the constitutive relation $p=P(v)$ as a pointwise
function, we regard it as a map of fields,
\[
  P:\mc A\subset L^1(\Omega)\to L^1(\Omega)  \com{via}
  p = P\circ v:\Omega\to\RR,
\]
and in order to extend weak* solutions, we need only extend this to a
map $\wh P$ defined on positive Radon measures.  Since pressure
vanishes at vacuum, this extension is easily accomplished using the
Lebesgue decomposition theorem.  To avoid unphysical solutions, we
impose a condition which we call \emph{consistency of the medium}, and
which states that the density and pressure must vanish whenever a
vacuum is present; although this can be regarded as an entropy-type
condition, it is distinct from the usual entropy condition which
degenerates to an equality at vacuum.  We refer to a vacuum weak*
solution which satisfies consistency of the medium as a \emph{natural
  vacuum solution}.  In our framework the entropy and entropy flux are
also regarded as maps on $L^1$ which are similarly extended to
positive Radon measures.  The entropy production is calculated to
be a measure which is supported on shocks, and which is required to be
negative.  This again agrees with the entropy condition for $BV$ weak
solutions.

Once we have defined natural vacuum solutions to \cref{psyst} that
include Dirac masses which account for vacuums of finite extent, we
present a few detailed examples.  These are natural vacuum solutions
but not weak solutions, and our explicit description of the solutions
and calculations of norms clearly demonstrates the advantages gained
by treating the test functions and integration by parts abstractly and
implicitly in the spaces $W^{1,q}_{w*}(0,T;X^*)$.

We next describe the straight-forward extension of our results to the
full $3\times3$ equations of gas dynamics in a Lagrangian frame.  We
again define an extension of the pressure and specific internal energy
to the positive Radon measures, by declaring that the pressure and
internal energy vanish at vacuum.  We then define a weak* solution and
the corresponding entropy condition as would be expected.

As a final application, we extend our results to the equations of
one-dimensional elasticity,
\[
  \del_tu - \del_xv = 0, \qquad
  \del_tv - \del_x\tau(u) = 0,
\]
where $u$, $v$ and $\tau(u)$ are the strain, velocity and stress,
respectively; we assume that $\tau'(u)>0$, with a softening response,
$\tau''(u)<0$.  Here we reproduce results of Giesselmann and
Tzavaras~\cite{GT}, in which they introduce so-called
\emph{slic}-solutions to study crack formation and resolve an apparent
paradox of nonuniqueness of solutions found in~\cite{Spector}.
Following~\cite{GT}, we study the onset of fracture, which we
represent as a Dirac mass in the strain.  To do so, we obtain the
natural extension of the stress to Dirac masses, namely
\[
  \wh\tau(w_0\,\delta_{x_0}) = L_\tau\,w_0\,\delta_{x_0},
  \com{where} L_\tau := \lim_{u\to\infty}\frac{\tau(u)}{u}.
\]
Extending the stress and the energy allows us to define weak*
solutions, and a brief analysis reveals that weak* solutions admitting
a crack are defined if and only if $L_\tau = 0$; however, none of
these solutions are entropic.  These are the same conclusions as those
of~\cite{GT}, but our results significantly extend the one-dimensional
results of~\cite{GT}, because their analysis applies to the single
example of a solution provided in~\cite{Spector}, while ours hold for
any crack in a weak* solution.  In~\cite{GT}, slic solutions are
obtained as limits of mollified approximations, and their calculation
of a single example requires several integrations and error estimates.
In contrast, with our approach the mollification and integration by
parts is abstract, and we are able to work directly with measures,
leading to a direct and exact development without the need for error
estimates.

The paper is arranged as follows: in \cref{sec:prelims}, we set
notation and recall the definition and properties of weak* integrable
functions and the Gelfand integral, developed in our earlier
paper~\cite{MY16}.  Next we recall the definition of weak* solutions
to conservation laws, and specifically to gas dynamics \cref{psyst},
and extend this definition to include vacuums.  We derive generalized
Rankine-Hugoniot jump conditions and discuss the entropy condition,
while showing that it remains an identity at the vacuum.  In
\cref{sec:exs} we present some detailed examples of natural vacuum
solutions which are not weak solutions.  \Cref{sec:gd} briefly
describes the extension of our methods to the full system of gas
dynamics, and in \cref{sec:elast} we consider the onset and
propagation of fractures in one-dimensional elasticity.

\section{Preliminaries}
\label{sec:prelims}

We begin by setting notation and recalling the Gelfand integral and
related notions which are necessary to define weak* solutions of
systems of conservation laws.  For simplicity we work in a single
space dimension.  We refer the reader to~\cite{MY16} for a more
detailed discussion and proofs of quoted results.

\subsection{Banach spaces}

Given a vector space $X$ with norm $\|\cdot\|_X$, we denote its dual
by $X^*$, and recall
\[
  \big\|\phi\big\|_{X^*} :=
  \sup_{x \in X, x \neq 0} \frac{\<\phi,x\>}{\|x\|_X}\,.
\]
We denote the $n$-fold product by
$X^n := \big\{ (x_1,x_2,\dots,x_n): \, x_i \in X\big\}$, and equip it
with the ``Euclidean'' norm
\[
  \big\|x\big\|_{X^n} :=
  \Big(\sum_{i=1}^n \big\|x_i\big\|_X^2\Big)^{1/2}\,.
\]
It follows that if we define the action of $\phi\in(X^*)^n$ on $x\in X^n$
by
\[
  \<\phi,x\> := \sum_{i=1}^n \<\phi_i,x_i\>,
\]
then we can write $(X^*)^n = (X^n)^*$.  In particular, any statements
on scalar valued function spaces $X = X(\Omega) = \{f:\Omega\to \RR\}$
extend naturally to vector-valued functions $X^n = \{F:\Omega\to
\RR^n\}$.

We recall the hierarchy of spaces that are most useful for us: first,
fixing an open bounded $\Omega\subset\subset\RR$, we have the inclusions
\[
  BV(\Omega) \subset L^\infty(\Omega) \subset
  L^p(\Omega)\subset L^1(\Omega).
\]
Next, any $f\in L^1(\Omega)$ generates a measure $\mu_f = \iota(f)$,
given by
\[
  \iota(f)(E) = \mu_f(E) := \int_E f\;dx, \com{for each}
  E\in \mc B(\Omega),
\]
so we regard $\iota(L^1(\Omega))\subset M(\Omega)$, the set of Radon
measures on $\Omega$; moreover, we have
\[
  \big\|\iota(f)\big\|_{M(\Omega)} = \big|\mu_f\big|(\Omega)
  = \int_\Omega |f|\;dx = \big\| f \big\|_{L^1(\Omega)}.
\]
Note that for any $f\in L^1(\Omega)$, $\iota(f)\ll\lambda$, that is
$\iota(f)$ is absolutely continuous with respect to Lebesgue measure,
and indeed, $f=\frac{d\mu_f}{d\lambda}$ is the Radon-Nikodym
derivative of $\iota(f)$.  On the other hand, by the Lebesgue
decomposition theorem, any Radon measure $\mu\in M(\Omega)$ can be
uniquely decomposed into absolutely continuous and singular parts,
\[
  \mu = \mu_c + \mu_s \com{with}
  \mu_c \ll \lambda  \com{and}
  \mu_s \perp \lambda,
\]
and moreover $\frac{d\mu_c}{d\lambda} \in L^1(\Omega)$.  We thus
define the map
\begin{equation}
  \Pi:M(\Omega)\to L^1(\Omega) \com{by}
  \Pi(\mu) := \frac{d\mu_c}{d\lambda} \in L^1(\Omega),
\label{Pi}
\end{equation}
the Radon-Nikodym derivative of the absolutely continuous part of
$\mu$.  It then follows that
\[
  \Pi\circ\iota(f) = f \com{for} f\in L^1(\Omega),
\]
while also
\begin{equation}
  \iota\circ\Pi(\mu) = \mu_c \com{for}
  \mu_c + \mu_s =: \mu \in M(\Omega),
\label{iota}
\end{equation}
so that $\iota\circ\Pi:M(\Omega)\to M(\Omega)$ is projection onto the
absolutely continuous part of the measure.

Recall that the Radon measures form the dual of $C_0$:
that is, regarding $C_0(\Omega)$ as the closure of
$C_c^\infty(\Omega)$ under the sup-norm, we can regard
$M(\Omega)=C_0(\Omega)^*$ under the action
\[
  \< \mu, \varphi\> = \int_\Omega \varphi(x)\;\mu(dx),
  \quad \varphi \in C_0(\Omega),
\]
and it is not difficult to verify that
$\big\|\mu\big\|_{C_0(\Omega)^*} = \big\|\mu\big\|_{M(\Omega)}$.

\begin{definition}
\label{def:DX}
We say that $f\in X^*$ has an $X^*$-valued distributional derivative,
written $\DD_xf \in X^*$, if, for all $\phi\in C_c^\infty(\Omega)\subset X$,
we have
\[
  \big|\<f,\phi'\>\big| \le C\,\|\phi\|_X,
\]
where we recall $C_c^\infty(\Omega)$ is dense in $X$, and in this case we
define $\DD_xf$ by
\[
  \< \DD_xf,\phi \> := - \<f,\phi'\>.
\]
\end{definition}

Finally, recall that $BV(\Omega)$ is the set of functions whose
distributional derivative $\DD_x f$ is in $L^1$:
\[
  \|f\|_{BV} = \sup\sum\big|f(x_{i+1})-f(x_i)\big|
  = \int_\Omega |\DD_x f|\;dx = \big\| \DD_x f \|_{L^1},
\]
the supremum and sum being taken over finite ordered partitions.

We can combine the above together with product spaces, using
inclusions as necessary, to get the following hierarchy of spaces:
\begin{equation}
  BV(\Omega)^n \subset L^\infty(\Omega)^n \subset
  L^1(\Omega)^n \subset M(\Omega)^n = {C_0(\Omega)^*}^n,
\label{spaces}
\end{equation}
where these functions take values in $\RR^n$.  Moreover, since $\Omega
\subset \subset \RR$ is arbitrary, these inclusions extend to locally
bounded functions on all of $\RR$,
\begin{equation}
  BV_{loc}^n \subset (L_{loc}^\infty)^n \subset
  (L_{loc}^1)^n \subset M_{loc}^n.
\label{locspaces}
\end{equation}

\subsection{The Gelfand integral}

We next recall the definition and calculus of the Gelfand integral,
which we need to define weak* solutions.  Again we refer the reader to
\cite{MY16} for more details and proofs of statements.  We briefly
discuss different ways to integrate functions mapping to an abstract
Banach space, namely the Bochner integral and Gelfand integral.

The Bochner integral of $f:[0,T]\to X$ is obtained by approximating
functions by simple functions.  The function $f$ is strongly
measurable, or Bochner measurable, if $f^{-1}(E)$ is measurable for
each measurable $E\subset X$.  The integral of a simple measurable
function $f$ is defined in the usual way,
\[
  \int \sum u_i \, \mc X_{E_i}(t)\; dt = \sum u_i\,\lambda(E_i) \in X,
\]
and $f$ is Bochner integrable if there is a sequence $\{h_n\}$ of
simple functions such that the Lebesgue integral $\int_0^T \| h_n -
f\| \; dt\to 0$ as $n\to\infty$, and in this case we have
$\int_Ef\;dt = \lim\int_Eh_n\;dt$.

The Bochner integral requires strong measurability, which is not
always obvious in an abstract Banach space.  The Dunford integral is a
weak integral, defined using the functionals on $X$.  For our purposes
it is more convenient to use the Gelfand integral, which is defined
for functions $\phi$ which take values in the dual space $X^*$ of a
Banach space $X$.  The map $\phi:[0,T]\to X^*$ is weak* measurable if
$\<\phi(\cdot),\a\>:[0,T]\to\RR$ is Lebesgue measurable for all $\a\in
X$.  Two functions $\phi$ and $\psi$ are weak* equivalent if
$\<\phi(\cdot),\a\>=\<\phi(\cdot),\a\>$ for $\lambda$-almost all $t$.
It can be shown that any weak* measurable function $\phi$ is weak*
equivalent to a function $\wh\phi$ which is norm-measurable, by which
we mean the scalar function $\|\wh\phi(\cdot)\|_{X^*}$ is Lebesgue
measurable.  We will denote the weak* equivalence class of a weak*
measurable $\phi$ by $[\phi]$, and a norm-measurable representative by
$\wh\phi\in[\phi]$, although we will often abuse notation by simply
writing $\phi$ when there is no ambiguity.

The Gelfand integral is defined as follows.  Suppose that we are given
a weak*-measurable function $\phi:[0,T]\to X^*$, and suppose also that
\[
  \< \phi(\cdot), \a \> \in L^1(0,T) \com{for all} \a \in X\,.
\]
For a given Borel set $E$, we define the map $T_E:X\to L^1(0,T)$ by
\[
  T_E(\a) = \< \phi(\cdot), \a \>\,\chi_E(\cdot) \in L^1(0,T)\,.
\]
It is clear that $T_E$ is linear, and if $\a_n\to \a$ and $T_E(\a_n)\to
y$ in $L^1$, then by the Riesz-Fischer theorem, a subsequence
$T_E(\a_{n_k})(s)\to y(s)$ a.e., while also $T_E(\a_n)(s) \to \<
\phi(s), \a \>\,\chi_E(s)$ for all $s\in[0,T]$.  It follows that $y\in
L^1(0,T)$, so $T_E$ is closed, and further, by the closed graph
theorem, it is bounded, so we can write $\| T_E(\a) \|_{L^1} \le
\|T_E\|\,\|\a\|$ for all $\a\in X$.  Since integration is a bounded
linear operator of $L^1$ into $\RR$, it follows that the map
\[
  \a \mapsto \int_0^T T_E(\a)(s)\;ds = \int_E\< \phi(s), \a \>\;ds
\]
is a bounded linear functional on $X$, so defines an element of the
dual $X^*$.  This functional is the Gelfand integral of $\phi$ over
$E$, and we denote it by $\wint_E \phi(s)\;ds\in X^*$.  Thus the
Gelfand integral over a measurable set $E$ is that element of $X^*$
defined by the condition
\begin{equation}
  \label{G-int}
  \Big\langle \wint_E \phi(s)\;ds , \a \Big\rangle =
     \int_E \< \phi(s),\a\>\;ds \com{for all} \a\in X\,.
\end{equation}
Again it follows easily that if $\phi$ is Bochner integrable with
values in $X^*$, then it is Gelfand integrable and the integrals
coincide.

\subsection{Gelfand-Sobolev Spaces}

We now describe the $X^*$ valued Gelfand $L^q$ spaces, for $1\le
q\le\infty$.  Given a weak* equivalence class $[\phi]$ of Gelfand
integrable functions, set
\[
  \qn{[\phi]}_q := \inf\big\{ \|g\|_{L^q(0,T)}\  :\
  \|{\wh\phi}(t)\|\leq g(t)\ \textrm{$\lambda$-a.e.} \big\}\,,
\]
where $\wh\phi\in[\phi]$ is a norm-measurable element of the
equivalence class.  It follows that $\qn{\cdot}_q$ is a norm, and we let
$L^{q}_{w*}(0,T;X^*)$ be the space of equivalence classes $[\phi]$ of
finite norm,
\[
  L^{q}_{w*}(0,T;X^*) := \big\{ [\phi] :
     \qn{[\phi]}_q < \infty \big\}\,.
\]

It is not difficult to show that $L^q_{w*}(0,T;X^*)$ is a Banach space
and that the trivial inclusion of the Bochner $L^q$ space in the
Gelfand $L^q$ space
\[
L^q(0,T;X^*) \subset L^q_{w*}(0,T;X^*) \com{via} f \mapsto [f],
\]
is a norm-preserving isomorphism.  Moreover, if
$\wh\phi\in[\phi]\in L^q_{w*}(0,T;X^*)$ is norm-measurable, then
$\|\wh\phi\|\in L^q(0,T)$ and
\[
  \qn{[\phi] }_q
  = \|\;\|\wh\phi(\cdot)\|\;\|_{L^q(0,T)}\,.
\]
It follows that if $\phi$ is Bochner integrable, then we can calculate
the Gelfand integral as a Bochner integral.

Now suppose that $\phi$, $\psi:[0,T]\to X^*$ are weak* integrable, so
that $[\phi]$, $[\psi]\in L^1_{w*}(0,T;X^*)$.  We say that $\psi$ is
the \emph{Gelfand weak derivative} or G-weak derivative of $\phi$,
written $\phi'(t)=\psi(t)$ or $[\phi']=[\psi]$, if
\begin{equation}
\begin{aligned}
  \wint_0^T \phi(t) \, \eta'(t)\; dt
   &= - \wint_0^T \psi(t)\, \eta(t)\; dt, \com{that is}\\
  \int_0^T \< \phi(t),\a\> \, \eta'(t)\; dt
   &= - \int_0^T \< \psi(t),\a\> \, \eta(t) \; dt,
\end{aligned}
\label{w*wd}
\end{equation}
for all $\a\in X$ and scalar functions $\eta \in C^{\infty}_c(0,T)$.

We now define the space $W^{1,q}_{w*}\big(0,T;X^*\big)$, for
$1 \leq q \leq \infty$, to be the set of weak* equivalence classes
$[\phi]\in L^q_{w^*}\big(0,T;X^*\big)$ with G-weak derivative
$[\phi'] \in L^q_{w*}\big(0,T;X^*\big)$, with norm
\[
  \qn{ [\phi] }_{W^{1,q}_{w*}(0,T;X^*)} :=
  \begin{cases}
    \big(\int_0^T (\|\wh\phi(t)\|^q+\|\wh{\phi'}(t)\|^q)\;dt \big)^{1/q} \,,
      & 1\leq q < \infty\\[2pt]
    \esssup_{t\in[0,T]}\big( \|\wh\phi(t)\| + \|\wh{\phi'}(t)\| \big)\,,
      &  q=\infty\,,
  \end{cases}
\]
for  norm-measurable representatives $\wh\phi$ and $\wh{\phi'}$.

If in addition, $\phi$ has values in some $Y\subset X^*$, then we
write $\phi \in W^{1,q}_{w*}\big(0,T;Y,X^*\big)$, that is we set
\[
  W^{1,q}_{w*}\big(0,T;Y,X^*\big)
   = \Big\{ \phi \in W^{1,q}_{w*}(0,T;X^*)\;:\;
       y(t) \in Y,\  t\in[0,T] \Big\}.
\]
Note that we do not assume that $Y$ is a subspace of $X^*$, because we
use the topology of $X^*$ throughout.

In \cite{MY16} we state and prove some basic calculus theorems for the
Gelfand integral, and the interested reader is referred there for
details.  We summarize the main points in the following theorem, which
collects parts of Theorems 3.5 and 3.7 of~\cite{MY16}.

\begin{theorem}
\label{thm:w*FTC}
If $f\in W^{1,q}_{w*}(0,T,X^*)$, then it has an absolutely continuous
representative $\ol f:[0,T] \to X^*$, which satisfies
\begin{equation}
\label{ftc}
  \ol f(t_2) - \ol f(t_1) = \wint_{t_1}^{t_2} f'(s)\;ds,
\end{equation}
for all $t_1$, $t_2\in[0,T]$.
Moreover, for all $\a \in W^{1,p}(0,T;X)$ strongly integrable, we have
the integration by parts formula
\begin{equation}\label{w*intbyparts}
  \int_{t_1}^{t_2} \< f'(t), \a(t) \> \, dt
  = \<\ol f(s),\ol{\a}(s)\> \Big|_{t_1}^{t_2} -
     \int_{t_1}^{t_2} \< f(t), \a'(t) \>\,dt\,.
\end{equation}
\end{theorem}

\section{Weak* solutions}
\label{sec:w*}

In~\cite{MY16}, the authors introduced the notion of weak* solutions
to a general system of hyperbolic conservation laws in one space
dimension.  Given such a system,
\begin{equation}
  \label{cl}
  \del_tU + \del_xF(U) = 0, \quad U(0,\cdot) = U_0,
\end{equation}
with $U$, $F(U):[0,T]\times\RR\to \RR^n$, recall that a distributional
solution is a locally integrable function $U(t,x)$ satisfying
\[
  \int_0^\infty\int_\RR \big(U(t,x)\,\del_t \varphi(t,x)
    + F(U(t,x))\,\del_x \varphi(t,x)\big)\;dx\;dt
  + \int_\RR U_0(x)\,\varphi(0,x)\;dx = 0,
\]
for all compactly supported test functions $\varphi$, and if in
addition $U$ is locally bounded, it is a weak solution.  We note that
the necessity of explicitly multiplying by test function and
integrating by parts means that calculations are unwieldy and often
error estimates must be employed when analyzing weak solutions.

On the other hand, when considering weak* solutions, we will treat the
conservation law~\cref{cl} as an ODE in an appropriate Banach space.
Indeed, we look at \cref{cl} directly and allow this to act linearly
on the Banach space $X$ which contains $C_c^{\infty}$ as a dense
subspace.  That is, for each $t$, we treat $U(t) = U(t,\cdot)$ and
$\DD_xF(U(t,\cdot))$ as living in $X^*$, and we regard \cref{cl} as an
ODE in $X^*$, so that
\begin{equation}
  \label{ode}
  U' + \DD_xF(U(t,\cdot)) = 0, \quad U(0) = U_0,
\end{equation}
for appropriately defined time derivative $U'$.  The critical issue
for us is to make sense of the nonlinear flux $F(U)$ and its
derivative in the space $X^*$.

We then say that
\[
  U \in 
      W^{1,q}_{w*}\big(0,T^{-};Y_{loc},X^*_{loc}\big)
\]
is a \emph{weak* solution} of the system \cref{cl} if
\[
  U' + \DD_xF(U(t)) = 0 \com{in} L^q_{w*}(0,T;X^*_{loc}),
\]
and if $\ol U(0) = U_0$ in $X^*_{loc}$, where $\ol{U}(t)$ is the continuous
representative of the weak* equivalence class, and where $U'$ is the
G-weak derivative of $U$.  Here $X^*_{loc}$ is understood in the usual
sense and we allow any $1\le q\le \infty$.

In our previous paper~\cite{MY16}, we used $X=C_0(\Omega)^n$, so that
$X^*=M(\Omega)^n$, and we took $Y = BV(\Omega)^n$.  In that paper we
studied the connections between weak* solutions and weak solutions,
and proved the following theorem.

\begin{theorem}\label{thm:weak}
Suppose $U \in W^{1,q}_{w*}\big( 0,T^-; BV_{loc}^n, M^n_{loc}\big)$ is
a weak* solution to the Cauchy problem \cref{cl}, with continuous
representative $\ol U$.  Then $\ol U$ is H\"older continuous as a
function into $L^1_{loc}(\RR;\RR^n)$, that is, $\ol U \in C^{0,\,
  1-1/q}(0,T^-;L^1_{loc})$ for $1\le q\le \infty$.  The function $\ol
U(t,x)$ is a distributional solution of the Cauchy problem \cref{cl}.
In particular, if $U$ is locally bounded, that is $U \in
L^{\infty}_{w*}\big( 0,T^-; L_{loc}^{\infty}(\RR;\RR^n)\big)$, then
$\ol U(t,x)$ is also a weak solution to the Cauchy problem \cref{cl}.
\end{theorem}

In the same paper, we showed that a distributional solution with
appropriate bounds is also a weak* solution, and in particular $BV$
weak solutions are weak* solutions.  As an immediate consequence, it
follows that the global weak solutions generated by Glimm's method,
front tracking, and vanishing viscosity, all of which have uniformly
bounded total variation, are all weak* solutions, and the uniqueness
and stability results of Bressan et.al.~hold unchanged in the
framework of weak* solutions.

\subsection{Application to Isentropic Gas Dynamics}

Because of the flexibilty provided by the choices of growth rate $q$
and spaces $Y$ and $X^*$, we regard weak* solutions as more general
than weak solutions.  Indeed, we will generalize weak* solutions to
include the vacuum in a Lagrangian frame, in which local boundedness
is lost and the specific volume is allowed to be a measure.

We work with the system of gas dynamics in a Lagrangian frame, namely
\begin{equation}
\label{psys}
\begin{aligned}
\del_t v - \del_x u &= 0, \\
\del_t u + \del_x p &= 0,
\end{aligned}
\end{equation}
in which the pressure $p$ is specified as a function of specific
volume $v$ by a constitutive relation of the form
\begin{equation}
  p = P(v), \com{that is}
  p(t,x) = P(v(t,x)),
\label{pofv}
\end{equation}
satisfying the appropriate properties: the most common such
constitutive law is that of an ideal gas, for which
$P(v)=A\,v^{-\gamma}$, $\gamma>1$.

It follows immediately that as long as $v$ remains $BV$, then a $BV$
weak* solution can be defined as above.  However, we want to allow
solutions which include vacuums, which are represented by Dirac masses
in a Lagrangian frame.  To do so, we simply allow the specific volume
$v(t)$ to be a Radon measure, which includes all Dirac masses.  We
note that the velocity $u$ remains $BV$, even when $v$ is unbounded
and includes Dirac masses.  We thus extend the target set $Y$ to
include Dirac masses in the first component, while still requiring
that the vector of conserved quantities remain in the set
$W^{1,q}(0,T;M^2_{loc})$.  In order for this extension to make sense,
we must extend the constitutive relation so that the pressure is
defined for any specific volume, which can now be a positive Radon
measure.

The constitutive relation expresses the thermodynamic pressure in
terms of the specific volume, as $p=P(v)$.  This extends naturally to
a map of functions,
\begin{equation}
  P:\mc A\subset L^1(\Omega)\to L^1(\Omega)  \com{via}
  p = P\circ v:\Omega\to\RR,
\label{Pfunc}
\end{equation}
where $\mc A$ is the domain of $P$, and allows us to close
\cref{psys}.  We now wish to extend this constitutive map to be
defined on Radon measures, and use this to define vacuum solutions of
\cref{psys}, which will include Dirac masses which represent vacuums.

Recalling the Lebesgue decomposition, in the notation of \cref{Pi},
\cref{iota}, we write the measure $\mu\in M_{loc}(\RR)$ as
\[
  \mu = \mu_c + \mu_s,\com{with}
  \mu_c \ll \lambda, \quad
  \mu_s \perp \lambda,
\]
so that for any Borel set $A$,
\[
  \mu(A) = \int_A \Pi(\mu)\;dx + \mu_s(A),
\]
where $\Pi(\mu) = \frac{d\mu_c}{d\lambda}\in L^1_{loc}(\RR)$ is the
Radon-Nikodym derivative of the absolutely continuous part of $\mu$.
Since the Lebesgue decomposition is unique, we extend the constitutive
function to be defined on positive measures $M_{loc}(\RR)_+$ by
\[
  \wh P:{M_{loc}(\RR)}_+ \to L^1_{loc}(\RR), \com{by}
  \wh P(\mu) = P(\Pi(\mu)),
\]
since pressure vanishes at vacuum.  That is, we declare that the
singular part of the specific volume makes no contribution to the
pressure.

When generalizing the specific volume to a positive measure, we use
the following notation: given $V\in M_{loc}$ and referring to
\cref{Pi}, \cref{iota}, we write
\[
  v := \Pi(V) \com{and}  \nu := V - \iota(v),
\]
so that $V = \iota(v) + \nu$, with $v\in L^1_{loc}$ and
$\nu\perp\lambda$.  It then follows that the (generalized) pressure is
\[
  \wh P(V) = \wh P(\iota(v)+\nu) = P(v),
\]
so that, as expected, the generalized pressure is the composition of
the pressure function with the Radon-Nikodym derivative $v$ of the
absolutely continuous part of the measure $V$.

As a first attempt at defining a solution with vacuum, we again take
$X$ to be the set of continuous test functions, $X = C_0(\Omega)^2$,
and we set
\[
  \wt Y_{loc} = M_{loc}(\RR)_+ \times BV_{loc}(\RR)
        \subset M_{loc}(\RR)^2 =: X^*_{loc},
\]
where $M_{loc}(\RR)_+$ denotes Radon measures that are (strictly)
positive on all open sets, so that
\[
  \mu \in {M_{loc}}(\RR)_+ \com{iff}
  \mu\big((a,b)\big) > 0 \quad \forall\  a<b.
\]

\begin{definition}
A \emph{vacuum weak* solution} of the $p$-system
\cref{psys} is a pair
\[
  (V,u) \in W^{1,q}_{w*}(0,T;\wt Y_{loc},X^*_{loc}),
\]
satisfying
\begin{equation}
\begin{aligned}
   &  V' - \DD_x u = 0 \\
   &  u' + \DD_x\wh P(V) = 0
\end{aligned} \qquad \text{in $L^q_{w*}(0,T;X^*_{loc})$,}
\label{Veq}
\end{equation}
where $'$ denotes the G-weak derivative.  When
solving a Cauchy problem, the Cauchy data $(V_0,u_0)$ must be taken on
in the space $X^*$ by the time-continuous representative $(\ol V,\ol u)$,
that is
\[
  \big(\ol V(0),\ol u(0)\big) = \big(V_0,u_0\big) \com{in} X^*_{loc}.
\]
\end{definition}

\subsection{Properties of  Solutions with Vacuum}

As in the general case of $BV$ weak* solutions, we immediately observe
that vacuum weak* solutions have some implicit regularity: first, the
solutions have an absolutely continuous representative $\big(\ol
V(t),\ol u(t)\big)\in X^*$.  Also, since the flux $(-u,p)$ has a
distributional derivative in $X^* = M_{loc}^2$, both $\ol u(t)$ and
$p(t) = \wh P(\ol V(t))$ are $BV$ functions (of material variable $x$)
for all~$t$.

\subsubsection{Evolution of Atomic Measures}

Next, recalling that $x$ is a material rather than spatial variable,
we show that vacuums are stationary in a Lagrangian frame.

\begin{lemma}
\label{lem:dirac}
A nontrivial continuous Dirac measure is stationary: that is, a
measure
\[
  \mu:(a,b) \to M(\Omega)  \com{of the form}
  \mu = w(t)\,\delta_{X(t)} \in M(\Omega),
\]
with $w\ne 0$ and $X:(a,b)\to\Omega$, is continuous on the interval
$(a,b)$ if and only if $w(t)$ is continuous and $X(t)$ is constant on
$(a,b)$.
\end{lemma}

\begin{proof}
Recalling that $\|\delta_x\|_{M}=1$, it follows easily that for $x$,
$y\in\Omega$, and $\alpha$, $\beta\in\RR$,
\[
    \|\alpha\,\delta_x - \beta\,\delta_y \|_{M(\Omega)}
    = (|\alpha|+|\beta|) \, \11_{ \{ x \neq y \} } + |\alpha-\beta|\, \11_{ \{ x = y \} },
\]
where $\11_E$ is the indicator function on $E$.

It follows immediately that if $w(t)$ is continuous on $(a,b)$, then
so is the stationary measure $\mu(t) = w(t)\,\delta_{X_0}$, for any
$X_0\in\Omega$.

Similarly, for $t$, $t_0 \in (a,b)$, we have
\begin{equation}
\begin{aligned}
\label{mudiff}
  \big\|\mu(t)-\mu(t_0)\big\|_{M(\Omega)}
  = \big(|w(t)|+|w(t_0)|\big) \; \11_{ \{ X(t) \neq X(t_0) \}}
     +|w(t)-w(t_0)|\; \11_{ \{ X(t) = X(t_0) \} },
\end{aligned}
\end{equation}
both terms being non-negative.  Now if $\mu\in M(\Omega)$ is
continuous at $t_0$, then
\[
  \|\mu(t) - \mu(t_0)\|_{M(\Omega)} \to 0, \com{as}
  t\to t_0,
\]
so, since $|w(t_0)|\ne 0$, \cref{mudiff} implies both
\[
  \lim_{t\to t_0}  \11_{ \{ X(t) \neq X(t_0) \}}=0 \com{and}
  \lim_{t\to t_0} |w(t)-w(t_0)|\; \11_{ \{ X(t) = X(t_0) \}}=0.
\]
It follows that, given any $\epsilon>0$, there exists $\eta>0$ such
that $X(t)=X(t_0)$ for all $t \in (t_0-\eta,t_0+\eta)$, and moreover
\[
  |w(t)-w(t_0)| = \big\|\mu(t)-\mu(t_0)\big\|_{M(\Omega)} < \epsilon
  \com{for} |t-t_0|<\eta.
\]
Since $t_0$ is an arbitrary point in $(a,b)$, $w(t)$ is continuous on
$(a,b)$.  Finally, let $(c,d)\subset(a,b)$ be the maximal interval for
which $X(t)=X(t_0)$ for all $t \in (c,d)$.  If $c>a$, find another
$\epsilon_1$ so that $X(t) = X(c)$ for $t\in(c-\epsilon_1,c+\epsilon_1)$ to obtain
a contradiction; this implies $c=a$.  Similarly, $d=b$ and the result
follows.
\end{proof}

Note that in other topologies such as the Wasserstein distance used in
mass transfer problems, continuity need not imply that singular
measures are stationary.

\subsubsection{Evolution of Unbounded Maps}

We next show that integrable functions $f\in W^{1,q}_{w*}$ which are
unbounded blow up on stationary sets, consistent with vacuums being
stationary in a material coordinate.

To this end, let $\Omega=(c,d) \subset \RR$ and $X \in
C^1\big([a,b],\Omega\big)$, so that the curve
\begin{equation}\label{Ccurve2}
  \mc{C}=\Big\{(x,t)\;:\; t \in [a,b],\  x=X(t) \Big\}
  \subset [a,b] \times \Omega,
\end{equation}
and let $\gamma>0$ be such that
\[
  \gamma < \sup_{t \in [a,b]}\big(\min\big\{|X(t)-c|,\;|X(t)-d|\big\}\big).
\]
Also suppose that the function $f(t,x):(a,b)\times\Omega \to
(0,\infty)$ is continuous at each point of the set $\big( (a,b) \times
\Omega \big)\backslash \mathcal{C}$, that the possibly infinite
one-sided limits $\lim_{x \to X(t)^\pm} f(t,x)$ exist for each $t \in
(a,b)$, and that for some $1 < q \leq \infty$, the map
\begin{equation}\label{fconds}
  t\to f(t ,\cdot) \in W^{1,q}_{w*}\big(a,b ; L^1(\Omega), M(\Omega) \big).
\end{equation}

Denote the sets on which $f$ is unbounded by
\[
  S_{\infty}^\pm = 
  \Big\{t \in [a,b]\;:\;\lim_{x \to X(t)^\pm} f(t,x) = \infty \Big\}
  \com{and}
  S_{\infty} = S_{\infty}^- \bigcup S_{\infty}^+\,.
\]
We first show that $f$ is almost uniformly unbounded on the set
$S_\infty$, in the sense of~\cite{Zakon}.

\begin{lemma}
\label{lem:essunifconv}
The sets $S_{\infty}^+$, $S_{\infty}^-$, and $S_{\infty}$ are
measurable, and for any $\eta>0$, there are measurable sets
$A^\pm \subset S_\infty^\pm$, with
\[
  \lambda(S_{\infty}^-\backslash A^-)<\eta \com{and}
  \lambda(S_{\infty}^+\backslash A^+)<\eta,
\]
such that for every $m \in \NN$, there exists $\delta_m > 0$ such that
\[
  \essinf_{t \in A^\pm} f\big(t,X(t)\pm\e\big) > m \com{for all}
  0<\e<\delta_m\,.
\]
\end{lemma}

\begin{proof}
For $\e \in (0,\gamma)$, the functions
\[
  g^-_{\e}(t) = f(t,X(t)-\e) \com{and}
  g^+_{\e}(t) = f(t,X(t)+\e)
\]
are defined and continuous on all of $[a,b]$.

We have
\[
   S_{\infty}^\pm =
   \Big\{t \in (a,b):  \lim_{\e \to 0^+} g_{\e}^\pm(t) = \infty \Big\},
\]
so we can write this as
\[
   S_{\infty}^\pm =
   \bigcap_{k=1}^\infty\left\{\bigcup_{N=\lceil 1/\gamma\rceil}^{\infty}
   \bigcap_{n=N}^{\infty}\Big\{t \in [a,b]\;:\; g_{\frac{1}{n}}^\pm(t)>k
   \Big\} \right\},
\]
and continuity of $g^\pm_{\frac1n}$ yields measurability of $S_\infty^\pm$.

Now take any $\eta>0$.  By assumption $g_{\e}^->0$ on $(a,b)$ for
every $\e\in(0,\gamma)$, so we can write
\[
  S_{\infty}^\pm = \Big\{ t \in(a,b)\;: \;
    \lim_{\e \to 0^+}\frac{1}{g_{\e}^\pm(t)} = 0 \Big\}\,.
\]
By Zakon~\cite{Zakon}, there exist measurable sets $A^\pm \subset
S_{\infty}^\pm$, with $\lambda(S_{\infty}^\pm\backslash A^\pm)<\eta$,
such that for every  $m \in \NN$ there exists $\delta_m>0$ such that
\[
  \esssup_{t \in A^\pm} \frac{1}{g_{\e}^\pm(t)} < \frac{1}{m}
  \com{for all} 0<\e < \delta_m,
\]
and the proof follows.
\end{proof}

We next show that if the discontinuity $X(t)$ is non-stationary, then
$f$ is bounded almost everywhere along $\mc C$.

\begin{lemma}
\label{lem:posspeedcurve}
Let $\sigma>0$ and suppose that the curve $\mathcal{C}$ given in
\cref{Ccurve2} satisfies $X'>\sigma$ on $(a,b)$.  If $f$ satisfies the
conditions \cref{fconds} given above, then
\[
  \lambda(S^+_{\infty})=\lambda(S^-_{\infty})=\lambda(S_{\infty})=0\,.
\]
The same conclusion holds if $X'<-\sigma$ on $(a,b)$.
\end{lemma}

\begin{proof}
We shall obtain a contradiction by constructing a sequence of test
functions $\psi_m\in W^{1,\infty}\big(a,b; C_0(\Omega)\big)$ for which
one side of the integration by parts formula \cref{w*intbyparts} is
unbounded, while the other remains bounded.

Without loss of generality, we assume that that
$\lambda(S^-_{\infty})>0$ and $X'>\sigma>0$ on $(a,b)$.  According to
\cref{lem:essunifconv}, there exists a set $A^-$ with
$\lambda(A^-)>\frac{1}{2} \lambda(S_{\infty}^-)$, such that for every
$m \in \NN$ there exists $0<\delta_m<\gamma$ such that
\begin{equation}
  \label{fest}
  \essinf_{t \in A^-} f(t,X(t)-\e) > m \com{for all} 0 < \e < \delta_m\,.
\end{equation}

Let $\varphi:[0,\infty)\to \RR$ be a $C^1$ monotone function such that
$\varphi(x) = 1$ for $x \le \frac18$, $\varphi(x) = 0$ for $x \ge
\frac78$, $\varphi'(x) = -\frac32$ for $\frac14 \le x \le \frac34$,
and $-\frac32\le \varphi'(x)\le 0$ elsewhere.  For each $m \in \NN$,
define
\[
  \phi_m(x) = \begin{cases}
      \varphi(-x/\delta_m), & x \le 0, \\
      \varphi(x/\gamma), & x \ge 0,
  \end{cases}
  \com{and}
  \psi_m(t,x) = \phi_m(x-X(t)),
\]
so that $\psi_m \in C^1(\mc U)$.  Moreover, $\psi_m(t,\cdot)$ and
$\del_t\psi_m(t,\cdot)$ are in $C_0(\Omega)$ for each $t \in (a,b)$,
so the map $t \to \psi_m(t,\cdot)$ belongs to $W^{1,\infty}\big(a,b;
C_0(\Omega)\big)$.

We now use $f$ and $\psi_m$ in the integration by parts formula
\cref{w*intbyparts}: first,
\[
\begin{aligned}
  -\int_a^b \< f(t), \psi'_m(t)\> \; dt
  &= \int_a^b X'(t)\int_{c}^{X(t)} f(x,t)\,\phi'_m(x-X(t)) \;dx\;dt \\
  &\qquad  + 
  \int_a^b X'(t)\int_{X(t)}^{d} f(x,t)\,\phi_m'(x-X(t)) \;dx\;dt\\
   & =: I_1 + I_2\,.
\end{aligned}
\]
By changing variables and using Fubini's Theorem and
\cref{fest}, we get
\[
\begin{aligned}
  I_1 & \geq 
   \int_a^b\sigma\int_{X(t)-3\,\delta_m/4}^{X(t)-\delta_m/4}
     f(t,x)\,\frac{3}{2\,\delta_m}\; dx \; dt \\
   &= \frac{3\,\sigma}{2\,\delta_m} 
          \int_{-3\,\delta_m/4}^{-\delta_m/4}
           \int_a^b f(t,X(t)+\e)\;dt \; d\e\\
   &\ge \frac{3\,\sigma}{4}\,m\,\lambda(A^-)
   \ge \frac{3\,\sigma}{8}\,m\,\lambda(S_\infty^-).
\end{aligned}
\]
Next, setting $\ol{\sigma}=\sup_{t \in (a,b)} X'(t)$ and using
$-\frac3{2\,\gamma} \le \phi_m'(x) \le 0$ for $x\ge 0$, we have
\[
\begin{aligned}
     0 \leq - I_2 & \leq \frac{3\,\ol{\sigma}}{2\,\gamma}
          \int_a^b \int_c^d f(x,t) \; dx\; dt \\
   & \leq \frac{3\,\ol{\sigma}}{2\,\gamma}\int_a^b \bigg(1+\Big(\int_c^d f(x,t) \, dx\Big)^q \bigg) \; dt\\
   &  = \frac{3\,\ol{\sigma}}{2\,\gamma}\Big((b-a)+ \|f\|_{L^q(a,b; L^1(\Omega))}\Big)\,,
\end{aligned}
\]
which is bounded.  Using $\|\phi_m\|_\infty=1$, we estimate the other
terms in \cref{w*intbyparts} by
\[
\begin{gathered}
  \<\ol{f},\ol{\psi}_m\>\Big|_a^b \le
  \big\|\ol{f}(b)\big\|_{L^1(\Omega)} +
  \big\|\ol{f}(a)\big\|_{L^1(\Omega)} , \com{and}\\
  \Big| \int \<f'(t),\psi_m(t)\> \; dt \Big| \le
  \big\|f'\big\|_{L^q_{w*}(a,b; L^1(\Omega))}\;(b-a)^{\frac{1}{p}},
\end{gathered}
\]
which are also bounded.  Since $m$ is arbitrary, we have a
contradiction and the lemma is proved.
\end{proof}

\begin{corollary}\label{cor:singsetprop}
For $\mathcal{C}$ as in \cref{Ccurve2} and $f$ as in \cref{fconds},
$f$ blows up on an essentially stationary set, that is
\begin{equation}\label{emb2}
  \lambda\big(S_{\infty}\backslash Z_X\big)=0, \com{where}
  Z_X  := \Big\{t \in [a,b]:  X'(t)=0 \Big\}\,.
\end{equation}
\end{corollary}

\begin{proof}
Without loss of generality, suppose that
\[
  \lambda \Big(S_{\infty} \bigcap \Big\{t \in [a,b]\;:\;X'(t)>0
  \Big\}\Big)>0.
\]
By continuity, we can find times $(t_1,t_2)\subset(a,b)$ and
$\sigma>0$ such that
\[
  \lambda\big( S_{\infty} \bigcap (t_1,t_2)\big)>0 
  \com{and} X'(t)>\sigma,
\]
for all $t \in (t_1,t_2)$, contradicting \cref{lem:posspeedcurve}.
\end{proof}

\subsubsection{Nonphysical Solutions}

Despite the regularity shown above, our definition is not yet
restrictive enough due to a large number of extraneous solutions which
satisfy our definition but are clearly inadmissible for physical
reasons: here we present an explicit example.

Given constants $u_-$, $u_+$, $v_0>0$ and $w_0\ge 0$, set
\begin{equation}
\label{nonphys}
  u(t,x) =
  \begin{cases}
    u_-, & x<0\\
    u_+, & x>0
  \end{cases},
\qquad
  V(t) = \iota(v_0) + (w_0 + [u]\,t)\,\delta_0,
\end{equation}
where as usual $[u] = u_+ - u_-$; it is clear that $\wh P(V) = p(v_0)$
and that the equation \cref{Veq} is satisfied in $X^* =
M_{loc}(\RR)^2$,  for $t<-w_0/[u]$ if $[u]<0$, or for all $t$
otherwise.  This solution represents a varying vacuum located at
$x=0$, adjacent on both sides to constant states with finite specific
volume $v_0$.  This is nonphysical because there is no
rarefaction between $v_0<\infty$, at which the pressure is positive,
and the vacuum, at which $p=0$, while there is no shock because
$p=P(v_0)$ on either side of $x=0$.  The physical solution is the
entropy solution of the vacuum Riemann problem, described below, or
the usual Riemann problem if $w_0=0$.

\subsection{Natural Vacuum Solution}

In view of the nonphysical examples \cref{nonphys}, it is clear that
our definition is not yet restrictive enough.  Moreover, entropy
considerations play no part here, because $[u]$ can be arbitrary.
Thus we need to rule out non-physical solutions without resorting to
the entropy condition.

The key observation here is that in \cref{nonphys}, we have allowed
vacuums, corresponding to stationary singular measures in the specific
volume, to occur while the projected specific volume $v = \Pi(V)\in
L^1_{loc}(\Omega)$ remains bounded, so that the gas does not rarefy
near the vacuum.  This is clearly unphysical and should be ruled out,
so we require that $v\to\infty$ as vacuum is approached.  We call this
property \emph{consistency of the medium}; it can also be interpreted
as a boundary condition induced by the vacuum.

We thus define the set of \emph{positive consistent measures},
\begin{equation}
  \label{cons}
    M_\infty = \big\{ \mu \in {M_{loc}}_+ \;:\;
     x \in \textrm{supp}(\mu_s) \implies
     \esslim_{y\to x}\Pi(\mu_c(y)) = \infty \big\},
\end{equation}
where we have again written $\mu = \mu_c + \mu_s$ using the Lebesgue
decomposition, and we set
\[
   Y_{loc} = M_\infty \times BV_{loc}\subset M^2_{loc} = X^*_{loc}.
\]

\begin{definition}
\label{def:NVT}
The pair $(V,u) \in W^{1,q}_{w*}(0,T;Y_{loc},X^*_{loc})$ is a \emph{natural vacuum
  solution} of the $p$-system \cref{psys}, if it satisfies \cref{Veq},
namely
\[
\begin{aligned}
  V' - \DD_x u &= 0,\\
  u' + \DD_x\wh P(V) &= 0.
\end{aligned}
  \com{in} L^q_{w*}(0,T;X^*_{loc}),
\]
A natural vacuum solution solves the Cauchy problem with Cauchy data
$(V_0,u_0)$ if the time-continuous representative $(\ol V,\ol u)$
satisfies
\[
  \big(\ol V(0),\ol u(0)\big) = \big(V_0,u_0\big) \com{in} X^*_{loc}.
\]
\end{definition}

\subsection{Rankine-Hugoniot Conditions}

Since $X^* = M_{loc}^2$ and \cref{Veq} is satisfied in $X^*$, the
distributional derivatives $\DD_xu$ and $\DD_xp$ are both measures,
which in turn implies that $u(t)$ and $p(t)$ are $BV$ functions of
$x\in\Omega$ for a.e.~$t$.  Thus, for a.e.~$t$, both $u(t)$ and $p(t)$
have well-defined left and right limits for each $x$ with countably
many jumps.  In order to obtain appropriate jump conditions, we assume
that the solution has a single isolated discontinuity located at
$x=X(t)$. 

Specifically, suppose there is an open set $\mc U = \{(t,x): t \in
(a,b),\ x\in (c,d)\}$ and that $X(t) \in (c,d)$ for all $t \in [a,b]$,
so that $\mathcal{C}= \{(t,x) : x=X(t),\ t\in(a,b)\}\subset \mc U$.
We assume that $u$ and $V$ are $C^1$ functions of $(t,x)$ on the open
region $\mc U\backslash \mathcal{C}$.  In particular, $v = \Pi(V)$ is
finite and $p>0$ at any point of $\mc U\backslash\mathcal{C}$.

According to these assumptions, and since the discontinuity is
isolated, it follows that for $(t,x)\in\mc U$ , we can write
\begin{equation}
\begin{aligned}
  u(t,x) &= u_L(t,x)\,H\big(X(t)-x\big)
     + u_R(t,x)\,H\big(x-X(t)\big) \com{and}\\
  p(t,x) &= p_L(t,x)\,H\big(X(t)-x\big)
     + p_R(t,x)\,H\big(x-X(t)\big),
\end{aligned}
\label{up}
\end{equation}
where each of $u_L$, $u_R$, $p_L$ and $p_R$ are in $C^1(\mc U)$ and
$H$ is the Heaviside function.  We denote the jump in a quantity $g$
by
\[
  [g](t) := g_R\big(t,{X(t)+}\big)-g_L\big(t,{X(t)-}\big),
\]
so that both $[u]$ and $[p]$ are differentiable functions of $t$.
Since $V(t)$ is a Radon measure, which may contain a Dirac mass, we
assume it has the form
\[
  V(t) = \iota\Big(v_L(t,\cdot)\,H\big(X(t) - \cdot)\big)\Big)
   + \iota\Big(v_R(t,\cdot)\,H\big(\cdot - X(t))\big)\Big)
   + w(t)\,\delta_{X(t)},
\]
where $v_L(t,\cdot)$, $v_R(t,\cdot)\in L^1(\mc U)$ and $w$ are
differentiable functions of $t$, consistent with $V\in
W^{1,1}_{w*}(0,T;L^1_{loc},M_{loc})$, as in \cref{lem:dirac} and
\cref{cor:singsetprop}.  Note that $v_L$ and $v_R$ are generally
unbounded as $x\to X(t)$ so that $[v]$ is not necessarily defined.
However, since $p_L, \ p_R \in C^1(\mc U)$ and $p=P(v)=\wh P(V)$, left
and right limits of $v(t,\cdot)$ exist everywhere, although these may
be infinite on the curve $\mc C$.

With these assumptions, we now calculate the appropriate
derivatives and plug them in to \cref{Veq}.  Using the distributional
derivative, we have
\[
\begin{aligned}
  \DD_xp &= \iota\Big(\del_xp_L\,H\big(X(t)-\cdot\big)\Big)
     + \iota\Big(\del_xp_R\,H\big(\cdot-X(t)\big)\Big)
     + [p]\,\delta_{X(t)}, \com{and} \\
  \DD_xu &= \iota\Big(\del_xu_L\,H\big(X(t)-\cdot\big)\Big)
     + \iota\Big(\del_xu_R\,H\big(\cdot-X(t)\big)\Big)
     + [u]\,\delta_{X(t)},
\end{aligned}
\]
where these are to be interpreted as measures.  Next, we calculate
\[
  u' = \iota\Big(\del_tu_L\,H\big(X(t)-\cdot\big)\Big)
     + \iota\Big(\del_tu_R\,H\big(\cdot-X(t)\big)\Big)
     + [u](t)\,\big(-X'(t)\big)\,\delta_{X(t)},
\]
where, since the distributional $t$-derivative is a measure, it
coincides with the B-weak and G-weak derivatives.  Equating the
coefficients of the Dirac masses in $\cref{Veq}_2$
then yields the first Rankine-Hugoniot condition,
\begin{equation}
  \label{RH1}
  X'(t)\,[u] = [p],
\end{equation}
while away from the curve $x=X(t)$, the equation $\del_tu+\del_xp=0$
holds in the classical sense.

Our assumptions combined with \cref{cor:singsetprop} imply that the set
\[
\begin{aligned}
    S_{\infty} &= \big\{t \in (a,b)\;:\; p_L(t,X(t)-)=0 \com{or}
                                       p_R(t,X(t)+)=0 \big\}\\
  &= \big\{t \in (a,b)\;:\; v_L(t,X(t)-)=\infty \com{or}
                               v_R(t,X(t)+)=\infty \big\}
\end{aligned}
\]
satisfies
\begin{equation}\label{Sinf}
  \lambda\big( S_{\infty} \bigcap \big\{t\;:\;  X'(t) \neq 0 \big\}\big)=0.
\end{equation}
Thus if $X'(t)\ne0$, then the limits $v_L\big(t,{X(t)-}\big)$ and
$v_R\big(t,{X(t)+}\big)$ must be finite, and in this case the jump
$[v]$ makes sense.

Differentiating the measure $V(t)$, and using \cref{Sinf}, we get, for
almost every $t\in S_{\infty}$,
\begin{equation}
\label{Vpr1}
  V'(t) = \iota\Big(\del_tv_L\,H\big(X(t)-\cdot\big)\Big)
     + \iota\Big(\del_t v_R\,H\big(\cdot-X(t)\big)\Big)\\
     + w'(t)\,\delta_{X(t)},
\end{equation}
while for almost every $t\in (a,b) \backslash{S}_{\infty}$,
we have
\begin{equation}
\begin{aligned}
  V'(t) &= \iota\big(\del_tv_L\,H\big(X(t)-\cdot\big)\big)
     + \iota\big(\del_t v_R\,H\big(\cdot-X(t)\big)\big)\\
&\quad +v_L\big(t,{X(t)-}\big)\,X'(t)\,\delta_{X(t)}
     - v_R\big(t,{X(t)+}\big)\,X'(t)\,\delta_{X(t)}\,\\
    &\quad  + w'(t)\,\delta_{X(t)}
     + w(t)\,X'(t)\,\DD_x\delta_{X(t)}\,.
 \end{aligned}
\label{Vpr}
\end{equation}

It follows that $V'$ in general is a distribution, defined by its
action on test functions.  In a weak* solution, $V'(t)\in M_{loc}$
must be a bounded measure, so that the coefficient of the last term of
\cref{Vpr} necessarily vanishes,
\begin{equation}
  \label{wX0}
  w(t)\,X'(t) = 0,
\end{equation}
consistent with \cref{lem:dirac} above.  In addition, all coefficients
of $\delta_{X(t)}$ in \cref{Vpr} must necessarily be bounded, so we can
combine \cref{Vpr1} and \cref{Vpr} into
\[
  V'(t) = \iota\Big(\del_tv_L\,H\big(X(t)-\cdot\big)\Big)
     + \iota\Big(\del_tv_R\,H\big(\cdot-X(t)\big)\Big)
     + \big(w'(t) - [v]\,X'(t)\big)\,\delta_{X(t)},
\]
where we have used the convention that
\[
  \big([v] \, X' \big) (t):=0 \com{for all} t \in S_{\infty}.
\]
Equating the measures in $\cref{Veq}_1$ yields the
Rankine-Hugoniot condition
\begin{equation}
  \label{RH2}
  [u] = w'(t) - X'(t)\,[v],
\end{equation}
while away from the curve $x=X(t)$, the equation $\del_tv-\del_xu=0$
again holds in the classical sense.

We can regard the three conditions \cref{RH1}, \cref{RH2} and \cref{wX0} as
generalized Rankine-Hugoniot conditions suitable for vacuum solutions,
which naturally extend the usual conditions, as follows.
First suppose that $w(t)>0$, which corresponds to the presence of a
vacuum of spatial width $w(t)$.  By continuity, this condition
persists in an interval $(t-\e,t+\e)$.  Also, by \cref{wX0}, we have
$X'(t) = 0$, so $X(t) =: X_0$ is constant in this interval,
consistent with \cref{lem:dirac}, and the natural vacuum condition
gives $v(t,X(t)\pm)=\infty$, so that
\begin{equation}
  w'(t) = [u], \com{while also}
  p(t,X_0\pm) = 0,
\label{vacRH}
\end{equation}
so that the pressure vanishes at vacuum as expected, while the spatial
expansion rate of the vacuum is the jump in velocity.
Now suppose that $X'(t)\ne 0$, so the discontinuity is not
stationary.  Again by continuity this holds in $(t-\e,t+\e)$, and
we conclude from \cref{wX0} that $w(t)=w'(t)=0$.  By our earlier
remark, both $v_L$ and $v_R$ remain finite, and \cref{RH1}, \cref{RH2}
reduce to the usual Rankine-Hugoniot conditions,
\begin{equation}
  X'(t)\,[-v] = [u], \qquad X'(t)\,[u] = [p].
\label{RH}
\end{equation}

\subsection{Entropy Condition}

For smooth solutions, it is easy to derive an energy equation:
multiplying the first equation of \cref{psys} by $-p$, the second by
$u$, and adding gives the scalar equation
\[
  \del_t\big(\TS{\frac12}u^2\big) - p\,\del_tv + \del_x(u\,p) = 0.
\]
Thus, using the specific internal energy, which satisfies
\[
  \eps = E(v) := \int_v^\infty P(\ol v)\;d\ol v,
\]
we obtain the conservation of energy,
\[
  \del_t\big(\TS{\frac12}u^2 + \eps\big) + \del_x(u\,p) = 0,
\]
which in turn provides an entropy/flux pair for solutions with shocks.

As in \cref{Pfunc}, the internal energy $E$ can be regarded as a map of
functions,
\[
  E:\mc A\subset L^1(\Omega)\to L^1(\Omega)  \com{via}
  \eps = E\circ v:\Omega\to\RR,
\]
and we again extend this to the positive measures $M_{loc}(\Omega)_+$ by
\[
  \wh E:{M_{loc}(\Omega)}_+ \to L^1_{loc}(\Omega), \com{by}
  \wh E(\mu) = E(\Pi(\mu)).
\]

We now impose the entropy condition for a natural vacuum solution
$(V,u)$, namely, we require that the map
\[
  t \mapsto \big(\TS{\frac12}u^2 + \wh E(V)\big)
    \in W^{1,1}_{w*}(0,T;M_{loc},M_{loc}),
\]
and that the entropy production be non-positive,
\begin{equation}
  \big(\TS{\frac12}u^2 + \wh E(V)\big)'
     + \DD_x\big(u\,\wh P(V)\big) \le 0
  \com{in} L^1_{w*}(0,T;M_{loc}),
\label{ent}
\end{equation}
both terms being interpeted as a measure.

In regions where the natural vacuum solution is differentiable, the
entropy inequality is satisfied as an equality.  On the other hand, if
the solution is discontinuous on an isolated curve $X(t)$, we again
describe the solution using \cref{up}.  Calculating the derivatives of
the measures as in the previous section, the measure in \cref{ent}
becomes
\[
  -X'(t)\,\big([\TS{\frac12}u^2] + [\eps]\big)\,\delta_{X(t)}
   + [u\,p]\,\delta_{X(t)},
\]
where the absolutely continuous part cancels because the solution is
classical where it is differentiable, and our entropy condition thus
becomes
\begin{equation}
  -X'(t)\,\big([\TS{\frac12}u^2] + [\eps]\big) + [u\,p] \le 0.
\label{entineq}
\end{equation}
Again there are two possibilities: first, if $X'(t) = 0$, then
\cref{vacRH}, \cref{RH} imply that $[u\,p]=0$ and the entropy condition is
satisfied as an equality, reflecting the fact that there is no shock.

On the other hand, if $X'(t)\ne0$, we use the identity
\[
  [g_1\,g_2] = \ol g_1\,[g_2] + [g_1]\,\ol g_2,
  \com{with} \ol g := \frac{g_R + g_L}2,
\]
together with \cref{RH}, to write
\[
\begin{aligned}
  -X'(t)\,\big([\TS{\frac12}u^2] + [\eps]\big) + [u\,p]
  &= -X'(t)\,\big(\ol u\,[u] + [\eps]\big) + \ol u\,[p] + [u]\,\ol p\\
  &= -X'(t)\,\big([\eps] + [v]\,\ol p\big).
\end{aligned}
\]
Now note that
\[
  [\eps] = E(v_R) - E(v_L) = \int_{v_R}^{v_L}P(\ol v)\;d\ol v, \com{and}
  [v]\,\ol p =
   - \int_{v_R}^{v_L}\TS{\frac12}\big(P(v_L)+P(v_R)\big)\;d\ol v,
\]
and so if $P(v)$ is convex, as is usually the case, then the entropy
inequality holds provided $X'(t)\,(v_R-v_L) > 0$.  This in turn
expressed the well-known fact that the pressure is greater behind the
shock, and reduces to Lax's shock condition.

\section{Examples of Natural Vacuum Solutions}
\label{sec:exs}

By way of example we present some explicit examples of natural vacuum
solutions, which are not weak solutions because of the presence of
vacuums, but which clearly extend the class of $BV$ weak solutions.
Before writing down the examples we introduce a convenient variable
and describe the elementary waves of the system.

\subsection{Symmetric Variables}

As in \cite{}, it is convenient to describe the solutions using a
nonlinear change of thermodynamic variable, which in turn simplifies
the description of waves.  Recall that the usual costitutive relation
is given by a pointwise function $P:(0,\infty)\to(0,\infty)$
expressing the pressure in terms of specific volume, $p=P(v)$, with
the properties
\begin{equation}
  P'(v) < 0, \quad
  \lim_{v\to\infty}P(v) = 0, \com{and}
  \int_1^\infty \sqrt{-P'(v)}\;dv < \infty.
\label{Pass}
\end{equation}
These conditions express hyperbolicity of the system, pressureless
vacuum, and possibility of vacuum formation, respectively.
Hyperbolicity is the condition that allows for forward and backward
nonlinear waves which propagate with (absolute) Langrangian sound
speed $C(v) := \sqrt{-P'(v)}$.  It is clear that each of these
properties is satisfied for an ideal gas, which has constitutive
function $P(v) = A\,v^{-\gamma}$, for $\gamma>1$.

We define the auxiliary function
\[
  H:(0,\infty)\to(0,\infty) \com{by}
  H(v) := \int_v^{\infty} C(\ol v) \; d\ol v,
\]
and introduce the symmetric variable $h$, which defines $v = v(h)$ by
\[
  v = v(h) := H^{-1}(h), \quad
  p = p(h) := P(H^{-1}(h)), \com{and}
  c = c(h) := C(H^{-1}(h)).
\]
Our assumptions \cref{Pass} imply that $H$ is monotone decreasing and
thus invertible, and that $H$ vanishes as $v\to\infty$, so the vacuum
is characterized as $h=0$, a bounded state.  In addition, we have
$p(0)=0$ and $c(0)=0$, and we calculate
\[
\begin{aligned}
  \frac{dv(h)}{dh} &= \Big(\frac{dH}{dv}\Big)^{-1} = \frac{-1}{C(v)}
   = \frac{-1}{c(h)},  \com{and}\\[2pt]
  \frac{dp(h)}{dh} &= \frac{dP}{dv}\,\Big(\frac{dH}{dv}\Big)^{-1}
   = \frac{-C^2(v)}{-C(v)} = c(h).
\end{aligned}
\]

Using the symmetric variable $h$, we rewrite the $p$-system \cref{psys}
as
\[
  \del_t v(h) - \del_x u = 0, \qquad
  \del_t u + \del_x p(h) = 0,
\]
and we now regard the unknowns as $(h,u)$.  When the solution is
differentiable, we can write it in the symmetric quasilinear form
\[
  \del_t h + c(h)\,\del_x u = 0, \qquad
  \del_t u + c(h)\,\del_x h = 0,
\]
and it is clear that the Riemann invariants are $u\pm h$, so we can write
the diagonal form
\begin{equation}
  \del_t (u+h) + c(h)\,\del_x (u+h) = 0, \qquad
  \del_t (u-h) - c(h)\,\del_x (u-h) = 0.
\label{RI}
\end{equation}

Finally, recalling that the specific internal energy is given by
\[
  E(v) = \int_v^\infty P(\ol v)\;d\ol v,
\]
we set $\eps(h) = E(H^{-1}(h))$ and calculate
\[
  \frac{d\eps(h)}{dh} = \frac{dE}{dv}\,\Big(\frac{dH}{dv}\Big)^{-1}
   = \frac{-P(v)}{-C(v)}
   = \frac{p(h)}{c(h)},  \com{so}
  \eps(h) = \int_0^h \frac{p(\ol h)}{c(\ol h)}\;d\ol h,
\]
and again for differentiable solutions we get the entropy equation
\[
  \del_t \big( \TS{\frac12}u^2 + \eps(h)\big)
   + \del_x \big( u\,p(h)\big) = 0,
\]
which yields the usual entropy inequality for shocks.

By way of example, it is a straight-forward calculation to describe a
$\gamma$-law gas, for which $P(v) = A\,v^{-\gamma}$, fully in terms of
symmetric variables: up to rescaling by a constant, we have
\begin{equation}
  c(h) = h^\beta, \quad
  v(h) = \frac{h^{1-\beta}}{\beta-1}, \quad
  p(h) = \frac{h^{1+\beta}}{\beta+1}, \com{and}
  \eps(h) = \frac{h^2}{2(\beta+1)},
\label{ideal}
\end{equation}
where the constant $\beta := \frac{\gamma+1}{\gamma-1}>1$.

\subsection{Elementary Waves}

There are two types of elementary waves, namely shocks and simple
waves.  Shocks satisfy the Rankine-Hugoniot conditions
\[
  X'(t)\,[-v] = [u], \qquad X'(t)[u] = [p],
\]
which yields
\[
  X'(t) = \pm \sigma := \pm\sqrt{[p]/[-v]}, \qquad
  [u] = \pm \sigma\,[-v],
\]
where $\sigma>0$ is the absolute shock speed.  For definiteness, we
assume that the pressure $p=P(v)$ is convex, so the entropy condition
implies that the pressure is greater behind the shock.  Thus for a
backward shock, $[-v] = v_L-v_R = \big|[-v]\big|>0$, while for a
forward shock $[-v] = -\big|[-v]\big|<0$.  In either case, the states
are related by
\begin{equation}
\begin{gathered} [u] = u_R - u_L = - \sigma\,\big|[-v]\big| = -
\sqrt{[p]\,[-v]}, \\ \sigma = \sqrt{[p]/[-v]}, \com{and} X'(t) =
\pm\sigma.
\end{gathered}
\label{shock}
\end{equation}

A simple wave is a $C^1$ solution of the quasilinear system with
one-dimensional image, so we can take $u = u(h)$, say.  Plugging this
into the Riemann invariant equations \cref{RI}, we get
\[
  \big(u'(h)+1\big)\,\Big(\del_th + c(h)\,\del_xh\Big) = 0  \com{and}
  \big(u'(h)-1\big)\,\Big(\del_th - c(h)\,\del_xh\Big) = 0.
\]
Thus, along the forward and backward characteristics
\[
  \frac{dx}{dt} = \pm c(h(t,x)), \com{we have}
  \frac{dh}{dt} = 0 \com{and}  u'(h) = \pm 1,
\]
respectively.  In particular the characteristics are straight lines,
$x - x_0 = \pm c(h)\,\big(t-t_0\big)$, on which we have
$u(t,x) = \pm h(t,x) + K$.  That is, we describe the simple wave by
\begin{equation}
  \label{simple}
  h(t,x) = c^{-1}\Big(\pm\frac{x-x_0}{t-t_0}\Big), \quad
  u(t,x) \mp h(t,x) = u_* \mp h_*,
\end{equation}
where $(h_*,u_*)$ is a reference state, typically adjacent to the
wave.  Here $(t_0,x_0)$ is a reference point for the individual
characteristic, which will generally depend on the value of $h$; if
the point $(t_0,x_0)$ is fixed, it is the \emph{center} of the wave.
The wave is \emph{compressive} or \emph{rarefactive} if the absolute
wavespeed $c$ decreases or increases from behind the wave to ahead,
respectively.  In particular, a centered compression focusses in
future time, and a centered rarefaction focusses in past time.  We
note that a simple wave may appear adjacent to the vacuum, if $h\to0$
across the wave, and the corresponding characteristics approach the
boundary of the vacuum, characterized by $x=X_0$ constant.

\subsection{Collapse of a Vacuum}

Our first example shows the collapse of a vacuum state.  We consider a
compressive vacuum with adjacent forward and backward compressions,
all of which are centered so that they focus at the origin.  There are
thus no shocks for $t\le0$, and for positive times the solution is
resolved by solving a Riemann problem.  The setup is graphically
illustrated in \cref{fig:collapse}, in which two characteristic
pictures are shown: on the left we show the Lagrangian material frame
in which we work, and on the right the Eulerian spatial frame.

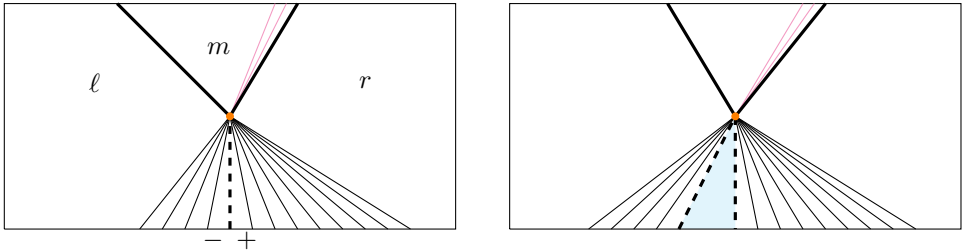
\begin{figure}[htb]
\begin{center}
\begin{tikzpicture}[scale=1.5]
  \draw[thin] (-2,-1) -- (2,-1) -- (2,1) -- (-2,1) -- (-2,-1);
  \draw[magenta!50] (0.4,1) -- (0,0) -- (0.5,1);
  \draw[very thick,dashed] (0,-1) -- (0,0);
  \draw[very thick] (-1,1) -- (0,0) -- (0.6,1);
  \foreach \cnt in {1,...,4}
     { \draw (-\cnt*0.2,-1) -- (0,0) -- (\cnt*0.2,-1); }
  \foreach \cnt in {5,...,8}
     { \draw (0,0) -- (\cnt*0.2,-1); }
  \fill[orange] (0,0) circle (1pt);
  \node at (-1.2,0.3) {$\ell$};
  \node at (1.2,0.3) {$r$};
  \node at (-0.1,0.6) {$m$};
  \node at (-0.15,-1.1) {$-$};
  \node at (0.15,-1.1) {$+$};
\end{tikzpicture}
\qquad
\begin{tikzpicture}[scale=1.5]
  \fill[cyan!10] (-0.5,-1) -- (0,0) -- (0,-1);
  \draw[thin] (-2,-1) -- (2,-1) -- (2,1) -- (-2,1) -- (-2,-1);
  \draw[magenta!50] (0.6,1) -- (0,0) -- (0.7,1);
  \draw[very thick,dashed] (-0.5,-1) -- (0,0) -- (0,-1);
  \draw[very thick] (-0.6,1) -- (0,0) -- (0.8,1);
  \foreach \cnt in {1,...,4}
     { \draw (-0.5-\cnt*0.2,-1) -- (0,0) -- (\cnt*0.2,-1); }
  \foreach \cnt in {5,...,8}
     { \draw (0,0) -- (\cnt*0.2,-1); }
  \fill[orange] (0,0) circle (1pt);
  \node at (0.15,-1.1) {\vphantom{$+$}};
\end{tikzpicture}
\caption{Centered collapse of a vacuum: Lagrangian and Eulerian
frames}
\label{fig:collapse}
\end{center}
\end{figure}

Referring to \cref{fig:collapse}, we choose states subscripted by
$\ell$, $m$ and $r$, together with velocities $u_-$ and $u_+$ adjacent
to the vacuum.  We can specify four data, say $u_-$, $u_+$, $h_\ell$
and $h_r$, and we require that $\Delta u:=u_+-u_-<0$, which ensures
the vacuum collapses.  Also, without loss of generality, we assume
$h_r\ge h_\ell$.  The remaining states are then deduced by \cref{shock}
or \cref{simple}, joining the various states by the corresponding
elementary wave.  Thus we have
\[
  u_\ell = u_- + h_\ell,  \com{and}  u_r = u_+ - h_r,
\]
and the state $(h_m,u_m)$ is found be resolving the Riemann problem,
see~\cite{}.  Since $u_r - u_\ell < 0$, there are two cases: either
$h_m>h_r$ (two shocks out) or $h_r\ge h_m>h_\ell$ (one shock out).

We use \cref{shock}, \cref{simple} to write down the solution explicitly:
for $t<0$, we have
\[
 h(t,x) =
  \begin{cases} h_\ell, &\\ c^{-1}(\frac xt), &\\ c^{-1}(\frac x{-t}),
&\\ h_r, &
  \end{cases} \quad u(t,x) =
  \begin{cases} u_\ell, &\quad x \le c(h_\ell)\,t\\ u_\ell - h_\ell +
c^{-1}(\frac xt), &\quad c(h_\ell)\,t\le x<0\\ u_r + h_r -
c^{-1}(\frac x{-t}), &\quad 0<x\le -c(h_r)\,t\\ u_r, &\quad
-c(h_r)\,t\le x
  \end{cases},
\]
and these in turn determine
\[
  p(t,x) = p(h(t,x)), \quad v(t,x) = v(h(t,x)), \com{and}
  V(t) = \iota(v(\cdot,t)) + \Delta u\,t\,\delta_0.
\]

For the outgoing waves, there are two cases: first, if $h_m > h_r$,
there are two outgoing shocks, so
\[ h(t,x) =
  \begin{cases} h_\ell, &\\ h_m, &\\ h_r, &
  \end{cases} \quad u(t,x) =
  \begin{cases} u_\ell, &\quad x < -\sigma_\ell\,t\\ u_m, &\quad
-\sigma_\ell\,t < x < \sigma_r\,t\\ u_r, &\quad \sigma_r\,t < x
  \end{cases},
\]
and
\[
  p(t,x) = p(h(t,x)), \quad v(t,x) = v(h(t,x)), \com{and}
  V(t) = \iota(v(\cdot,t)),
\]
where $\sigma_\cdot$ is the (absolute) shock speed, given by
$\sigma = \sqrt{[p]/[-v]}$,

On the other hand, if $h_m \le h_r$, then the right outgoing wave is a
rarefaction, and
\[ h(t,x) =
  \begin{cases} h_\ell, &\\ h_m, &\\ c^{-1}(\frac x{t}), &\\ h_r, &
  \end{cases} \quad u(t,x) =
  \begin{cases} u_\ell, &\quad x < -\sigma_\ell\,t\\ u_m, &\quad
-\sigma_\ell\,t < x \le c(h_m)\,t\\ c^{-1}(\frac x{t}) - h_m, &\quad
c(h_m)\,t\le x\le c(h_r)\,t\\ u_r, &\quad c(h_r)\,t\le x
  \end{cases},
\] and $p$, $v$ and $V$ given as above.

Because $h$ is monotone across each wave, it is clear that each of
$h$, $c$, $p$ and $u$ has bounded variation as a function of $x$.
Also, the abstract argument shows that because the characteristic
and/or shock conditions hold everywhere, we have a weak* solution.  It
remains to check that $V$ is a well-behaved measure, $V \in
W^{1,\infty}(0,T;M_\infty)$.

To this end, we note that, if $t<0$ and $\Omega = [-a,b]$, say,
\[
  V = \iota(v) + \Delta u\,t\,\delta_0, \com{and}
  \|V\|_{M(\Omega)} = |V|\big([-a,b]\big),
\]
and since $v\to\infty$ as $x\to 0$, $V\in M_\infty$ as long as it is bounded
as a measure.  Since $h\ge 0$ and $\Delta u<0$, for $t<0$ we calculate
\[
\begin{aligned}
  \|V(t)\|_{M(\Omega)} &= \Delta u\,t +
      \int_{-a}^bv\big(h(t,x)\big)\;dx \\
   &= \Delta u\,t + v(h_r)\,\big(b-(-c(h_r)\,t)\big)
      + v(h_\ell)\,\big(c(h_\ell)\,t-(-a)\big)\\
   &\qquad\qquad{} +
      \int_{c(h_\ell)\,t}^0 v\big(c^{-1}(\TS{\frac{x}{t}})\big)\;dx
      + \int_0^{-c(h_r)\,t} v\big(c^{-1}(\TS{\frac{x}{-t}})\big)\;dx\\
   &= \Delta u\,t + b\,v(h_r) + a\,v(h_\ell)
        + t\,\big(v(h_\ell)\,c(h_\ell) + v(h_r)\,c(h_r)\big) \\
    &\qquad\qquad{} + t \int^0_{c(h_\ell)}v\big(c^{-1}(y)\big) \;dy
      + t \int^0_{c(h_r)} v\big(c^{-1}(y)\big)\;dy \\
   &= b\,v(h_r) + a\,v(h_\ell) + t\,\big(\Delta u - h_\ell - h_r\big),
\end{aligned}
\]
where we have used
\begin{equation}
  \label{vint}
\begin{aligned}
  v(h)\,c(h) + \int^0_{c(h)}v\big(c^{-1}(y)\big) \;dy
    &= v(h)\,c(h) + \int_h^\infty v(z)\,c'(z)\;dz\\
    &= - \int_h^\infty c(z)\,v'(z)\;dz = -h,
\end{aligned}
\end{equation}
having integrated by parts, and used $v(h)\,c(h)\to 0$ as $h\to 0$.

Similarly, for $t\ge 0$, if $h_m>h_r$, so two shocks emerge,
\[
\begin{aligned}
  \|V(t)\|_{M(\Omega)} &= \int_{-a}^bv\big(h(t,x)\big)\;dx \\
  &= v(h_\ell)\,\big(a-\sigma_\ell\,t\big)
   + v(h_m)\,\big(\sigma_r\,t+\sigma_\ell\,t\big)
   + v(h_r)\,\big(b-\sigma_r\,t\big)\\
  &= b\,v(h_r) + a\,v(h_\ell)
   + t\,\big(\sigma_\ell\,(v(h_m)-v(h_\ell))
        + \sigma_r\,(v(h_m)-v(h_r))\big),
\end{aligned}
\]
while for $h_r\ge h_m$, the right outgoing wave is a rarefaction and
\[
\begin{aligned}
  \|V(t)\|_{M(\Omega)} &= v(h_\ell)\,\big(a-\sigma_\ell\,t\big)
   + v(h_m)\,\big(c(h_m)\,t+\sigma_\ell\,t\big)\\
  &\qquad\qquad{} + \int_{c(h_m)\,t}^{c(h_r)\,t}
    v\big(c^{-1}(\TS{\frac{x}{t}})\big)\;dx
    + v(h_r)\,(b-c(h_r)\,t) \\
  &= b\,v(h_r) + a\,v(h_\ell)
    + t\,\Big(\sigma_\ell\,\big(v(h_m)-v(h_\ell)\big) + h_r - h_m\Big),
\end{aligned}
\]
again using \cref{vint}.
It is now clear that $\|V(t)\|_{M(\Omega)}$ is bounded, and indeed it
is Lipschitz, as expected.  Piecewise linear dependence on $t$ occurs
in this instance because all waves are centered, and scale invariance
implies rank one homogeneity.

\subsection{Centered Waves and the Vacuum}

Our next example consists of a collapsing vacuum between two centered
simple waves, these being centered at different points, with one being
a compression and the other a rarefaction wave, as illustrated in
\cref{fig:offctr}.  One can pose this as a Cauchy problem by
taking the trace of the solution at time $t=0$.  We choose the data as
in the previous case, so that the initial compression and vacuum
collapse at the same point.

For short times the solution contains the vacuum and five other waves:
first, adjacent to the vacuum are the focussing compression, and
the centered rarefaction; next, a shock and a centered rarefaction
emerge from the point of collapse of the vacuum; and finally, as the
shock interacts with the original rarefaction it changes strength and
a backwards compression is transmitted behind the shock, as drawn in
the figure.  Of course, at some later time this reflected compression
will collapse to form a shock, which will lead to the generation of
more (ever weaker) waves, in a process which continues indefinitely.

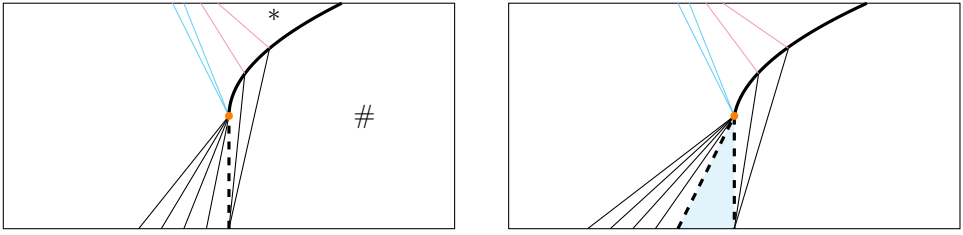
\begin{figure}[htb]
\begin{center}
\begin{tikzpicture}[scale=1.5]
  \draw[thin] (-2,-1) -- (2,-1) -- (2,1) -- (-2,1) -- (-2,-1);
  \draw[cyan!50] (-0.4,1) -- (0,0) -- (-0.5,1);
  \draw[very thick,dashed] (0,-1) -- (0,0);
  \draw[very thick,domain=0:1] plot (\x*\x,\x);
  \foreach \cnt in {1,...,4}
     { \draw (-\cnt*0.2,-1) -- (0,0); }
  \foreach \cnt in {1.25,2}
     { \draw (0,-1) -- (\cnt*\cnt*0.09,\cnt*0.3);
       \draw[magenta!50] (\cnt*\cnt*0.09,\cnt*0.3) -- (-0.5+\cnt*.2,1); }
  \fill[orange] (0,0) circle (1pt);
  \node at (1.2,0.0) {$\#$};
  \node at (0.4,0.9) {$*$};
\end{tikzpicture}
\qquad
\begin{tikzpicture}[scale=1.5]
  \fill[cyan!10] (-0.5,-1) -- (0,0) -- (0,-1);
  \draw[thin] (-2,-1) -- (2,-1) -- (2,1) -- (-2,1) -- (-2,-1);
  \draw[cyan!50] (-0.4,1) -- (0,0) -- (-0.5,1);
  \draw[very thick,dashed] (-0.5,-1) -- (0,0) -- (0,-1);
  \draw[very thick,domain=0:1] plot (\x*\x+0.175*\x,\x);
  \foreach \cnt in {1,...,4}
     { \draw (-0.5-\cnt*0.2,-1) -- (0,0); }
  \foreach \cnt in {1.25,2}
     { \draw (0,-1) -- (\cnt*\cnt*0.09+0.06*\cnt,\cnt*0.3);
       \draw[magenta!50] (\cnt*\cnt*0.09+0.06*\cnt,\cnt*0.3)
                           -- (-0.5+\cnt*.2,1); }
  \fill[orange] (0,0) circle (1pt);
\end{tikzpicture}
\caption{Vacuum Adjacent to Centered Waves: Lagrangian and Eulerian
frames}
\label{fig:offctr}
\end{center}
\end{figure}

The main issue in resolving the solution for short times is an exact
description of the states and trajectory of the shock wave; once we
know these, it is routine to describe the simple waves via
characteristics using \cref{simple}.  We briefly describe the process
for exactly resolving the shock wave before secondary interactions
occur.  For simplicity, we assume a $\gamma$-law gas, given by
\cref{ideal}.

The shock trajectory is a curve in the plane, and it is convenient to
parameterize it by the state $h$ ahead of the shock, which is also
part of the centered rarefaction in the data.  Thus the shock lies on
the curve $\big(x(h),t(h)\big)$, and using \cref{shock} for the
trajectory and \cref{simple} for the centered rarefaction, we have
\[
  \frac{dx}{dh} = \sigma(h)\,\frac{dt}{dh},  \com{while also}
  x(h) = c(h)\,t(h),
\]
where $c(h)=h^\b$ is the speed of the characteristic from the origin,
and $\sigma(h)$ is the shock speed at the point $\big(x(h),t(h)\big)$.
In particular, we need to show that the shock curve can be defined up
to the point of vacuum collapse, that is, that the limit exists as
$h\to 0+$.  Combining these relations yields the linear differential
equation
\[
  \frac{dc}{dh}\,t(h) + c(h)\,\frac{dt}{dh} =
  \sigma(h)\,\frac{dt}{dh},
\]
which we can solve to get
\[
  t = t_\#\,\textrm{exp}\Big(\int_{h_\#}^h\frac{c'(h)\;dh}{\sigma(h)-c(h)}\Big),
  \com{and} x(h) = c(h)\,t(h).
\]
It follows that provided the integral converges as $h\to 0$, we can
choose $t_\#$ so that the shock begins at the appropriate point.  Using
\cref{ideal}, we can write this integral as
\begin{equation}
  I = \int_{h_\#}^h\frac{c'(h)\;dh}{\sigma(h)-c(h)}
    = \b\,\int_{h_\#}^h\frac1{\sigma/c - 1}\;\frac{dh}h.
\label{int}
\end{equation}

We now consider the states on either side of the shock, using the
following notation: parameterizing the ahead (right) state by $h$, we
write the behind state as $h_b = z(h)\,h$, so $z(h)>1$ is defined to
be the ratio of behind state to ahead state.  Again using \cref{ideal},
we then write the jump across the state as
\[
\begin{aligned}
  [p] &= p(h_b) - p(h) = \frac{(z\,h)^{\b+1} - h^{\b+1}}{\b+1}
      = h^{\b+1}\,z^{\b+1}\,q_{\b+1}(z) \com{and}\\
  [-v] &= v(h) - v(h_b) = \frac{h^{1-\b} - (z\,h)^{1-\b}}{\b-1}
      = h^{1-\b}\,q_{\b-1}(z),
\end{aligned}
\]
where we have set
\[
  q_n(z) := \frac{1-z^{-n}}n.
\]
Using this notation, the shock relations \cref{shock} simplify to
\begin{equation}
  [u] = - h\,z^{\frac{\b+1}2}\,r(z) \com{and}
  [\sigma] = h^\b\,z^{\frac{\b+1}2}\,s(z) = c(h)\,z^{\frac{\b+1}2}\,s(z),
\label{us}
\end{equation}
where we have defined
\[
  r(z) := \sqrt{q_{\b+1}(z)\,q_{\b-1}(z)} \com{and}
  s(z) := \sqrt{q_{\b+1}(z)/q_{\b-1}(z)}.
\]

It remains to find $z(h)$, which will in turn completely determine the
shock trajectory and states.  We do this by exactly resolving the
interaction of the shock and centered interaction wave.  Referring to
\cref{fig:offctr}, we label states as follows: the reference ahead
state is $h_\#$, with corresponding behind state $h_*=z_\#\,h_\#$, and
the varying ahead state is $h$, with corresponding behind state $h_b =
z\,h$.  Across the shock, we have
\[
  u_\# - u_* = - h_\#\,z_\#^{\frac{\b+1}2}\,r(z_\#) \com{and}
  u_h - u_b = - h\,z^{\frac{\b+1}2}\,r(z),
\]
while the waves joining the other states are simple, so we use
\cref{simple} to write
\[
  u_\# - u_h = h_\# - h  \com{and}
  u_* - u_b = h_b - h_* = z\,h - z_\#\,h_\#.
\]
Eliminating $u$, we get
\[
  u_\# - u_b = h_\# - h - h\,z^{\frac{\b+1}2}\,r(z)
     = - h_\#\,z_\#^{\frac{\b+1}2}\,r(z_\#) + z\,h - z_\#\,h_\#,
\]
which simplifies to
\begin{equation}
  h\,\big(1 + z + z^{\frac{\b+1}2}\,r(z)\big)
  = h_\#\,\big(1 + z_\# + z_\#^{\frac{\b+1}2}\,r(z_\#)\big) =: A,
\label{hA}
\end{equation}
where $A$ is a reference constant.  This last relation determines
$z(h)$ implicitly, but we can work explicitly by changing variables:
it is clear that this equality is monotone in $z$, so that
$z\to\infty$ as $h\to 0$, and differentiating, we get
\[
  dh\,\big(1 + z + z^{\frac{\b+1}2}\,r(z)\big)
  + h\,\big(1 + z + z^{\frac{\b+1}2}\,r(z)\big)'\,dz = 0,
\]
so the integral in \cref{int} becomes explicit,
\[
  I = - \b\,\int_{z_\#}^z
    \frac{\big(1 + z + z^{\frac{\b+1}2}\,r(z)\big)'\;dz}
	{\big(1 + z + z^{\frac{\b+1}2}\,r(z)\big)
           \,\big(z^{\frac{\b+1}2}\,s(z)-1\big)}\,.
\]
Since the integrand is of order $z^{\frac{\b-1}2}/z^{\b+1} =
z^{-\frac{\b-3}2}$ for $z$ large, the integral converges as
$z\to\infty$ and the shock emerges from the collapse as required.
Moreover, using \cref{hA}, \cref{us} we write
\[
  h = \frac A{1 + z + z^{\frac{\b+1}2}\,r(z)} \com{and}
  \sigma = \frac{A^\b\,z^{\frac{\b+1}2}\,s(z)}
   {\big(1 + z + z^{\frac{\b+1}2}\,r(z)\big)^\b}
  = O(1)\,z^{\frac{1-\b^2}2},
\]
so we have both $\frac{dh}{dz}<0$ and $\frac{d\sigma}{dz}<0$, so
$\frac{d\sigma}{dh}>0$, and the shock trajectory is initially concave,
as drawn.

\subsection{Vacuum Riemann Problem}

For our final example, we introduce a generalization of the Riemann
problem, which allows for the presence of an embedded vacuum of finite
spatial width in the initial data.  That is, our data consists of
bounded constant left and right states $(h_\ell,u_\ell)$ and $(h_r,u_r)$,
together with an initial spatial width $w_0 \ge 0$ of a vacuum located
at $x=0$; a zero width $w_0=0$ reduces to the usual Riemann problem.
The solution of the vacuum Riemann problem then provides a building
block for the construction of general solutions which contain vacuums.

If $w_0>0$, then there is a vacuum in the solution, which must have
simple waves adjacent to it, and these must be centered at the origin,
so are rarefactions.  The left rarefaction connects $(h_\ell,u_\ell)$
to $(0,u_-)$, and the right rarefaction connects $(0,u_+)$ to
$(h_r,u_r)$, where $u_-$ and $u_+$ are the velocities at the edge of
the vacuum, and are given by
\[
  u_- = u_\ell + h_\ell \com{and} u_+ = u_r - h_r.
\]
The quantity
\[
  \Delta u := u_+ - u_- = u_r - h_r - u_\ell - h_\ell
\]
determines whether the vacuum is compressive or rarefactive: if
$\Delta u\ge 0$, the vacuum persists for all times $t\ge 0$, and the
solution is
\begin{equation}
\label{RPsol}
 h(t,x) =
  \begin{cases}
     h_\ell, &\\
     c^{-1}(\frac{-x}t), &\\
     c^{-1}(\frac xt), &\\
     h_r, &
  \end{cases} \
 u(t,x) =
  \begin{cases}
    u_\ell, &\quad x \le - c(h_\ell)\,t\\
    u_\ell + h_\ell - c^{-1}(\frac{-x}t), &\quad -c(h_\ell)\,t\le x<0\\
    u_r - h_r + c^{-1}(\frac x{t}), &\quad 0<x\le c(h_r)\,t\\
    u_r, &\quad c(h_r)\,t\le x
  \end{cases},
\end{equation}
and these in turn determine
\[
  p(t,x) = p(h(t,x)), \quad
  v(t,x) = v(h(t,x)) \com{and}
  V(t) = \iota(v(\cdot,t)) + \big(w_0 + \Delta u\,t\big)\,\delta_0.
\]
We note that the functions $u(t,x)$, $h(t,x)$ and $v=\Pi(V)$ are
self-similar, but as long as $w_0>0$, neither the solution nor data
are self-similar.

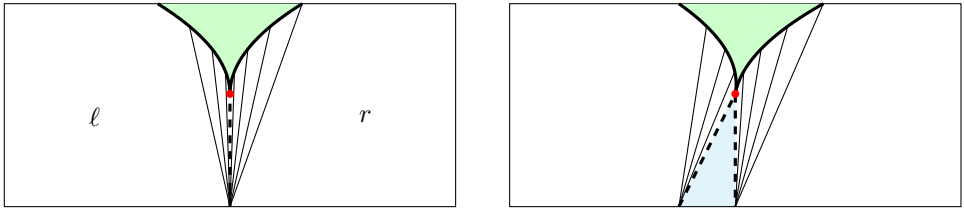
\begin{figure}[htb]
\begin{center}
\begin{tikzpicture}[scale=1.5]
  \draw[thin] (-2,-1) -- (2,-1) -- (2,0.8) -- (-2,0.8) -- (-2,-1);
  \draw[very thick,dashed] (0,-1) -- (0,0);
  \foreach \cnt in {1,2,3}
  {
    \draw (0,-1) -- (\cnt*\cnt*0.04,\cnt*0.2);
    \draw (0,-1) -- (-\cnt*\cnt*0.04,\cnt*0.2);
  }
  \draw (0,-1) -- (0.64,0.8);
  \node at (-1.2,-0.2) {$\ell$};
  \node at (1.2,-0.2) {$r$};
  \fill[green!20,domain=0:0.8] plot (-\x*\x,\x) -- (0,0.8) -- (0,0);
  \fill[green!20,domain=0:0.8] plot (\x*\x,\x) -- (0,0.8) -- (0,0);
  \draw[very thick,domain=0:0.8] plot (\x*\x,\x);
  \draw[very thick,domain=0:0.8] plot (-\x*\x,\x);
  \fill[red] (0,0) circle (1pt);
\end{tikzpicture}\qquad
\begin{tikzpicture}[scale=1.5]
  \fill[cyan!10] (-0.5,-1) -- (0,0) -- (0,-1);
  \draw[thin] (-2,-1) -- (2,-1) -- (2,0.8) -- (-2,0.8) -- (-2,-1);
  \draw[very thick,dashed] (-0.5,-1) -- (0,0) -- (0,-1);
  \fill[green!20,domain=0:0.8] plot (-\x*\x+0.175*\x,\x) -- (0.25,0.8) -- (0,0);
  \fill[green!20,domain=0:0.8] plot (\x*\x+0.175*\x,\x) -- (0.25,0.8) -- (0,0);
  \draw[very thick,domain=0:0.8] plot (\x*\x+0.175*\x,\x);
  \draw[very thick,domain=0:0.8] plot (-\x*\x+0.175*\x,\x);
  \foreach \cnt in {1,2,3}
  {
    \draw (-0.5,-1) -- (-\cnt*\cnt*0.04+0.035*\cnt,\cnt*0.2);
    \draw (0,-1) -- (\cnt*\cnt*0.04+0.035*\cnt,\cnt*0.2);
  }
  \draw (0,-1) -- (0.78,0.8);
  \fill[red] (0,0) circle (1pt);
\end{tikzpicture}
  \caption{Collapsing Vacuum Riemann Problem: Lagrangian and Eulerian
    frames}
  \label{fig:VRP}
\end{center}
\end{figure}

On the other hand, if the vacuum is compressive, so $\Delta u<0$,
\cref{RPsol} provides a solution for only finite times: indeed, the
vacuum that was initially in the data collapses at time $T=-w_0/\Delta
u$.  As the vacuum collapses, two shocks emerge from the point
$(T,0)$, and the solution evolves non-trivially in the region behind
these shocks, as shaded in \cref{fig:VRP}.  For short times after
collapse, this can be resolved as in the previous example with slight
modifications.

\section{Compressible Euler Equations}
\label{sec:gd}

With minor modifications, we can apply our methods to the full
$3\times3$ system of the compressible Euler equations in one space
dimension.  The system is obtained by using a more general
constitutive law, which satisfies the Second Law of Thermodynamics.
This introduces a second thermodynamic variable, so requires another
conservation law, which is conservation of energy.  The thermodynamic
quantities of interest are the specific volume $v$, pressure $p$,
internal energy $\eps$, specific entropy $s$ and temperature $T>0$, and
they are related by the Second Law,
\[
  T\,ds = d\eps + p\,dv,
\]
so, when using $v$ and $s$ as independent variables, we write
\[
  \eps = E(v,s), \quad
  p = - \frac{\del E}{\del v} \com{and}
  T = \frac{\del E}{\del s}.
\]
In a Lagrangian frame, the equations are
\begin{equation}
  \label{euler}
    v_t - u_x = 0, \qquad
    u_t + p_x = 0, \qquad
    (\TS{\frac12}u^2 + \eps)_t + (u\,p)_x = 0,
\end{equation}
and for smooth solutions we easily derive the entropy equation
$s_t = 0$.  The most familiar constitutive law is that of an ideal
polytropic gas,
\[
  E(v,s) = \frac A{\gamma-1}\,v^{1-\gamma}\,e^{s/c_v}, \quad
  P(v,s) = A\,v^{-\gamma}\,e^{s/c_v}, \com{with}
  \gamma > 1.
\]

Because both $\eps$ and $p$ vanish at vacuum, we can define natural
vacuum solutions for the Euler equations just as for the $p$-system.
Using \cref{cons}, we set
\[
   Y = M_\infty \times BV_{loc}\times BV_{loc}\subset M^3_{loc} = X^*,
\]
and say that a triple
\[
  (V,u,s) \in W^{1,q}_{w*}(0,T;Y,X^*),
\]
is a natural vacuum solution of the Euler equations \cref{euler}, if it
satisfies
\[
  V' - \DD_x u = 0, \qquad
  u' + \DD_x\wh P(V,s) = 0, \qquad
  \big(\TS{\frac12}u^2 + \wh E(V,s)\big)'
     + \DD_x\big(u\,\wh P(V,s)\big) = 0,
\]
in $X^*$ for almost all $t>0$.  The solution solves the Cauchy
problem, with Cauchy data $(V_0,u_0,s_0)$ if the time-continuous
representative $(\ol V,\ol u,\ol s)$ satisfies
\[
  \big(\ol V(0),\ol u(0),\ol s(0)\big)
  = \big(V_0,u_0,s_0\big) \com{in} X^*.
\]

We extend the jump conditions without difficulty; in particular, where
the solution is differentiable it satisfies the quasilinear form of
the equation, and at jumps the generalized Rankine-Hugoniot relations
hold, namely
\[
\begin{gathered}
  X'(t)\,[u] = [p], \qquad
  X'(t)\,[\TS{\frac12}u^2 + \eps] = [u\,p], \\
  [u] = w'(t) - X'(t)\,[v],  \com{and}
  w(t)\,X'(t) = 0,
\end{gathered}
\]
using the same notation as before.

As is well known, the entropy field $s$ is a linearly degenerate
contact field with vanishing characteristic speed, and across which
$[u]=[p]=0$.  Since the vacuum also propagates with zero speed, it
follows that the entropy $s$ can jump arbitrarily across a vacuum;
however, the occurrence of the vacuum is detected by any change in the
velocity $u$ across the jump: if $[u]=0$ and $v$ is finite, the jump
is a contact, while if $[u]\ne0$ and $X'(t)=0$, the jump is a vacuum
with expansion rate $w'(t) = [u]$.

The solution is an entropy solution if it satisfies
\[
  s' \le 0 \com{in} M_{loc},
\]
where we recall $s = s(t)$ is regarded as a measure, and this is the
G-weak derivative.  As in the $2\times2$ case, this measure is
supported only on shocks, and the entropy equality $s'=0$ holds in the
presence of vacuums, as long as no shocks are present.

One can write down explicit examples as above, with the addition of
entropy jumps where necessary, while noting that in a varying
solution, the interaction of a shock with any non-trivial solution
changes the trailing entropy field.

\section{Fracture in Elasticity}
\label{sec:elast}

The time-dependent displacement or \emph{motion} $y(t,x)$ of an
isentropic elastic material in one space dimension satisfies the
second-order nonlinear equation
\begin{equation}
  y_{tt} - \tau(y_x)_x = 0.
\label{2der}
\end{equation}
By introducing the \emph{strain} $u := y_x$ and \emph{velocity} $v =
y_t$, we write this as the $2\times2$ equations of elasto-dynamics,
\begin{equation}
  u_t - v_x = 0, \qquad
  v_t - \tau_x = 0,
\label{elast}
\end{equation}
which is closed by prescribing the \emph{stress} $\tau$ by a
stress-strain relation $\tau = \tau(u)$, which serves as a
constitutive function.  This system closely resembles the $p$-system,
and for smooth solutions admits an energy inequality, namely
\[
  \big(\TS{\frac12}v^2 + W(u)\big)_t - \big(v\,\tau(u)\big)_x = 0,
\com{where}
  W(u) := \int^u\tau(s)\;ds
\]
is the elastic energy.  We do not allow interpenetration of matter, so
we require $u>0$, and we assume that $\tau'(u) > 0$, which implies
hyperbolicity of the system.  In terms of energy, these are implied by
$W(0+)=\infty$ and convexity of $W(u)$, respectively.  In particular,
we assume the existence of a unique strain $u_0$ such that
\[
  u_0 > 0 \com{and}
  \tau(u_0) = W'(u_0) = 0,
\]
that is at which the stress vanishes and the energy is minimized;
without loss of generality we may also assume that the energy vanishes
there, $W(u_0)=0$.  We generally assume also that the material is
\emph{softening}, which means that $\tau''(u)<0$.

We are interested in extending the notion of weak* solutions to this
system, in order to understand the onset of fractures or cavities in
the material.  This will again be represented as a Dirac mass in the
strain, or a discontinuity in the motion.  Although the model will
break down before an actual fracture occurs, a consistent picture of
the behavior indicated by the model provides insights into the process
of crack initiation.

\subsection{Weak* Solutions}

Following our development for the $p$-system, we regard the
stress-strain relation as providing a map
\[
  \tau:\mc A\subset L^1(\Omega)\to L^1(\Omega),
\]
and we wish to extend this to a map
\[
  \wh \tau : M_{loc}(\Omega)_+\to M(\Omega).
\]

Assuming for now that this extension has been defined, we again choose
the set of test functions to be $X=C_0(\Omega)^2$, and we set
\[
  \wt Y = {M_{loc}}(\Omega)_+ \times BV_{loc}(\Omega)
        \subset M^2_{loc}(\Omega) = X^*.
\]

\begin{definition}
A weak* solution of \cref{elast} is a pair
$(U,v) \in W^{1,q}_{w*}(0,T;\wt Y,X^*)$ satisfying
\begin{equation}
  U' - \DD_x v = 0,\qquad v' - \DD_x\wh \tau(U) = 0,
  \com{in} L^q_{w*}(0,T;X^*).
\label{Uveq}
\end{equation}
When solving a Cauchy problem, the Cauchy data
$(U^0,v^0)$ must be taken on in the space $X^*$ by the time-continuous
representative $(\ol U,\ol v)$, that is
\[
  \big(\ol U(0),\ol v(0)\big) = \big(U^0,v^0\big) \com{in} X^*.
\]
\end{definition}

It remains to extend the stress-strain relation to Radon measures.  We
again follow our development for the $p$-system.  Writing
$U\in{M_{loc}}_+$, we write
\[
  u := \Pi(U) \com{and} \mu := U - \iota(u), \com{so that}
  U = \iota(u) + \mu,
\]
with $u\in L^1_{loc}$ and $\mu\perp\lambda$.  Moreover, the singular
measure $\mu$ is uniquely decomposed into singular continuous and
atomic parts, that is
\[
  \mu = \mu_s + \mu_a, \com{with} \mu_a = \sum w_i\,\delta_{x_i},
\]
where $w_i>0$ and $x_i$ are distinct.  It is natural to require the
generalized stress to satisfy
\[
  \wh\tau(U) = \iota\big(\tau(u)\big)
    + \sum \wh\tau\big(w_i\,\delta_{x_i}\big) \in M,
\]
so we need only extend $\tau$ to a single Dirac mass.

Let $\phi$ denote a standard mollifier,
\[
  \phi \in C^{\infty}_c(\RR), \com{with}
  \phi(\cdot) \ge 0 \com{and}
  \int \phi(x)\;dx = 1,
\]
and set $\phi_\epsilon(x) := \phi(x/\epsilon)/\epsilon$.  Recalling
that the action of the Dirac mass $\delta_{x_0}\in M$ on a continuous
function $g\in C_0$ can be written as
\[
  \<\delta_{x_0},g\> =
  \lim_{\epsilon\to0} \int \phi_\epsilon(x-x_0)\,g(x)\;dx = g(x_0),
\]
we define the action of $\wh\tau(w_0\,\delta_{x_0})$ by
\[
  \<\wh\tau(w_0\,\delta_{x_0}),g\> =
  \lim_{\epsilon\to0} \int
  \tau\big(u_0 + w_0\,\phi_\epsilon(x-x_0)\big)\,g(x)\;dx.
\]
Here we include the offset $u_0$ so that the integrand again has
compact support.  We calculate
\[
\begin{aligned}
  \lim_{\epsilon\to0} \int&
  \tau\big(u_0 + w_0\,\phi_\epsilon(x-x_0)\big)\,g(x)\;dx\\
  &= \lim_{\epsilon\to0} \int_{\{x: \, \phi_{\epsilon}>0\}}
  \frac{\tau(u_0 + w_0\,\phi_{\epsilon})}{u_0 +
    w_0\,\phi_{\epsilon}}\,
  \big(u_0 + w_0\,\phi_\epsilon\big)\,g(x)\;dx \\
  &{}= \Big( \lim_{u\to\infty}\frac{\tau(u)}{u}\Big)\,w_0\,g(x_0)\,,
\end{aligned}
\]
because $\lambda\big(\{x: \, \phi_{\epsilon}>0\}\big) \to 0$.
 Thus we have
\begin{equation}
  \begin{gathered}
  \wh\tau(\mu_a) = \sum\wh\tau(w_i\,\delta_{x_i})
  = \sum L_\tau\,w_i\,\delta_{x_i} = L_\tau\,\mu_a,\\
  \com{where}
  L_\tau := \lim_{u\to\infty}\frac{\tau(u)}{u},
  \end{gathered}
\label{tauhat}
\end{equation}
and this completes our definition of the extension.
We note that the softening condition $\tau''(u)<0$ implies that
$L_\tau<\infty$ exists and is finite.

Similarly, in order to extend the notion of entropy solutions, we need
to extend the potential energy so that it is also defined on measures.
As above, it suffices to define the extension $\wh
W(w_0\,\delta_{x_0})$, for $w_0>0$; having done so, and writing
\[
  U = \iota(u) + \mu_a + \mu_s, \quad
  u:=\Pi(U), \quad \mu_a = \sum\wh\tau(w_i\,\delta_{x_i}),
\]
where $\mu_a$ is singular atomic and $\mu_s$ singular continuous, we
set
\[
  \wh W(U) = \iota\big(W(u)\big) + \sum\wh W(w_i\,\delta_{x_i}).
\]
Exactly as for the stress above, we extend the energy via a mollifier,
\[
\begin{aligned}
   \<\wh W(w_0\,\delta_{x_0}),g\> &=
  \lim_{\epsilon\to0} \int
   W\big(u_0 + w_0\,\phi_\epsilon(x-x_0)\big)\,g(x)\;dx \\
  &{}= \Big( \lim_{u\to\infty}\frac{ W(u)}{u}\Big)\,w_0\,g(x_0),
\end{aligned}
\]
so we can write
\begin{equation}
  \begin{gathered}
  \wh W(\mu_a) = \sum\wh W(w_i\,\delta_{x_i})
  = \sum L_W\,w_i\,\delta_{x_i} = L_W\,\mu_a,\\
  \com{provided}
  L_W := \lim_{u\to\infty}\frac{W(u)}{u} < \infty,
  \end{gathered}
\label{What}
\end{equation}
and $\wh W(\mu_a) = \infty$ if $L_W = \infty$; note that because the
energy $W$ is convex, the (possibly infinite) limit $L_W$ always
exists.

Having extended the energy, we say that a solution satisfying
\cref{Uveq} is an entropy weak* solution if the entropy production is
non-positive,
\begin{equation}
  \label{entrW}
  \big(\TS{\frac12}v^2 + \wh W(U)\big)'
     - \DD_x\big(v\,\wh \tau(U)\big) \le 0
  \com{in} L^q_{w*}(0,T;M_{loc}),
\end{equation}
both terms being interpeted as Radon measures.

\subsection{Properties of Weak* Solutions}

Having defined weak* solutions, we now examine their properties and
develop conditions that allow for a consistent model of fracture
initiation.  To begin, we examine the jump conditions as we did for
the $p$-system.  Thus we assume that we have a solution which is
differentiable off of a discontinuity curve $x=X(t)$.

Following \cref{up}, we assume that the velocity $v$ and strain $U$ have
the form
\begin{equation}
\begin{aligned}
  v(t,x) &= v_L(t,x)\,H\big(X(t)-x\big)
     + v_R(t,x)\,H\big(x-X(t)\big) \com{and}\\
  U(t) &= \iota\Big(u_L(t,\cdot)\,H\big(X(t) - \cdot)\big)\Big)
   + \iota\Big(u_R(t,\cdot)\,H\big(\cdot - X(t)\big)\Big)
   + w(t)\,\delta_{X(t)},
\end{aligned}
\label{Uv}
\end{equation}
so that, according to \cref{tauhat}, we also have
\begin{equation}
  \wh\tau\big(U(t)\big) =
   \iota\Big(\tau\big(u_L(t,\cdot)\big)\,H\big(X(t) - \cdot)\big)\Big)
   + \iota\Big(\tau\big(u_R(t,\cdot)\big)\,H\big(\cdot - X(t))\big)\Big)
   + L_\tau\,w(t)\,\delta_{X(t)}.
\label{tauU}
\end{equation}
Differentiating, we get
\[
\begin{aligned}
 \DD_xv &= \iota\Big(\del_xv_L\,H\big(X(t)-\cdot\big)\Big)
     + \iota\Big(\del_xv_R\,H\big(\cdot-X(t)\big)\Big)
     + [v]\,\delta_{X(t)} \com{and} \\
  v' &= \iota\Big(\del_tv_L\,H\big(X(t)-\cdot\big)\Big)
     + \iota\Big(\del_tv_R\,H\big(\cdot-X(t)\big)\Big)
     + [v](t)\,\big(-X'(t)\big)\,\delta_{X(t)},\\
\end{aligned}
\]
where we have set $[g](t):=g_R\big(t,X(t)+\big)-g_L\big(t,X(t)-\big)$,
and similarly
\[
\begin{aligned}
  U' &= \iota\Big(\del_tu_L\,H\big(X(t)-\cdot\big)\Big)
     + \iota\Big(\del_tu_R\,H\big(\cdot-X(t)\big)\Big)\\
   &\qquad{} + [u](t)\,\big(-X'(t)\big)\,\delta_{X(t)}
     + w'(t)\,\delta_{X(t)} + w(t)\,X'(t)\,\DD_x\delta_{X(t)},\\
 \DD_x\wh\tau(U) &= \iota\Big(\del_x\tau(u_L)\,H\big(X(t)-\cdot\big)\Big)
     + \iota\Big(\del_x\tau(u_R)\,H\big(\cdot-X(t)\big)\Big)\\
    &\qquad{} + [\tau(u)]\,\delta_{X(t)} + L_\tau\,w(t)\,\DD_x\delta_{X(t)}.
\end{aligned}
\]
It follows immediately that in order for \cref{Uveq} to make sense in
$X^*$, the coefficients of $\DD_x\delta_{X(t)}$ must vanish, so that
\begin{equation}
  \label{vanish}
  w(t)\,X'(t) = 0  \com{and}
  L_\tau\,w(t) = 0.
\end{equation}
Assuming these conditions, we then again obtain a generalized jump
condition,
\begin{equation}
  w'(t) - X'(t)\,[u] = [v], \qquad
  - X'(t)\,[v] = [\tau(u)],
\label{tauRH}
\end{equation}
while the system holds in the classical sense where the solution is
differentiable.  It follows that in order to admit any fracture
solution in $W^{1,1}_{w*}(0,T;Y,X^*)$, we must have
\[
  L_\tau = \lim_{u\to\infty}\frac{\tau(u)}{u} = 0,
\]
and if this holds then any crack must necessarily be stationary,
$X'(t) = 0$.  These conditions are consistent with gas dynamics, for
which $p(\infty)=0$.

Next, in order to avoid spurious singular measure solutions, we need
to again impose a consistency of the medium condition: this is a
modelling condition dependent on physical effects at the crack after
the crack has been initiated.  If there is a nontrivial crack, it must
be stationary and by \cref{tauRH} we also have $[\tau] = 0$.  This means
that the stress, and so also the strain, has a single limit, say
\begin{equation}
  \tau\big(u(t,x)\big) \to \tau(u_*), \com{or}
  u(t,x) \to u_*, \com{as} x \to X(t),
\label{bdry}
\end{equation}
and is analogous to the natural vacuum condition.  For example, the
simplest condition is that once the crack has been initiated, it
imposes no extra force on the interior material, so that the crack
boundary is stress-free, $u_* = u_0$.

Finally, we consider the effect of entropy: again using \cref{Uv},
\cref{tauU}, and referring to \cref{entrW}, we write the entropy as
\[
  \wh\eta = \iota\Big(\big(\TS{\frac12}v_L^2 +
           W(u_L)\big)\,H\big(X(t)-\cdot\big)\Big) +
       \iota\Big(\big(\TS{\frac12}v_R^2 +
           W(u_R)\big)\,H\big(\cdot-X(t)\big)\Big)
       +  L_W\,w(t)\,\delta_{X(t)},
\]
while, using \cref{vanish}, the entropy flux becomes
\[
  - v\,\wh\tau(U) =
  - \iota\Big(\big(v_L\,\tau(u_L)\big)\,H\big(X(t)-\cdot\big)\Big)
  - \iota\Big(\big(v_R\,\tau(u_R)\big)\,H\big(\cdot-X(t)\big)\Big).
\]
Again differentiating, and using \cref{vanish}, the entropy production
\cref{entrW} becomes
\[
\begin{aligned}
  \big(\TS{\frac12}v^2 &+ \wh W(U)\big)'
     - \DD_x\big(v\,\wh \tau(U)\big) \\&=
   \iota\Big(\big(\del_t\big(\TS{\frac12}v_L^2 +
           W(u_L)\big)-\del_x\big(v_L\,\tau(u_L)\big)\big)
           \,H\big(X(t)-\cdot\big)\Big) \\
    &\quad{}+ \iota\Big(\big(\del_t\big(\TS{\frac12}v_R^2 +
           W(u_R)\big)-\del_x\big(v_R\,\tau(u_R)\big)\big)
           \,H\big(\cdot-X(t)\big)\Big)\\
  &\quad{} + \Big(-X'(t)\big[\TS{\frac12}v^2 +  W(U)\big]
           + L_W\,w'(t) - \big[v\,\tau(u)\big]\Big)\,\delta_{X(t)},
\end{aligned}
\]
where again $[g] = g_R\big(t,{X(t)+}\big) - g_L\big(t,{X(t)-}\big)$.
The first two terms cancel because the solution is differentiable, so
the entropy condition becomes non-positivity of the coefficient, so
\begin{equation}
  -X'(t)\big[\TS{\frac12}v^2 + W(U)\big]
           + L_W\,w'(t) - \big[v\,\tau(u)\big] \le 0.
\label{nonpos}
\end{equation}
In order to make sense of this inequality, we require that
$L_W<\infty$, which in turn implies
\[
  L_W = \lim_{u\to\infty}\frac{W(u)}{u}
      = \lim_{u\to\infty}\tau(u) =: \tau_\infty < \infty,
  \com{so also} L_\tau=0.
\]
Note that \cref{nonpos} is consistent with \cref{entineq} for a gas, for
which $p_\infty = 0$.

If there is no crack, then $w(t) = 0$, and \cref{nonpos} reduces to the
usual entropy condition for shocks.  On the other hand, if there is a
crack, so $w(t)>0$, then $X'(t)=0$ and, using \cref{tauRH}, \cref{bdry},
\cref{nonpos} simplifies as
\[
\begin{aligned}
  L_W\,w'(t) - \big[v\,\tau(u)\big]
    &= L_W\,w'(t) - [v]\,\tau(u_*) \\
    &= \big(\tau_\infty - \tau(u_*)\big)\,w'(t) \le 0.
\end{aligned}
\]
It follows that $w'(t) \le 0$, so that any crack satisfying the
entropy condition must be pre-existing with $w(0)>0$.  This says that
the elasticity of the material prevents crack formation in any solution.

In summary, we have shown the following.

\begin{theorem}
The space $W^{1,q}_{w*}(0,T;X^*)$ supports weak* solutions
which admit fractures if and only if the limit $L_\tau=0$.  However,
any nontrivial fracture fails to satisfy the entropy condition
unless it has finite size in the initial data.
\label{thm:elast}
\end{theorem}

\subsection{Comparison to Slic-solutions}

In \cite{GT}, Giesselmann and Tzavaras introduce a notion of
\emph{slic}-solution in order to study the formation of cavities and
fractures in dynamic elasticity.  In doing so they work primarily with
the second-order system \cref{2der}, namely
\[
  y_{tt} - \tau(y_x)_x = 0,
\]
and they study discontinuities in the motion $y$.  Their main goal is
to settle a question of nonuniqueness in earlier examples, in which
cavitating solutions apparently had a lower energy than smooth
solutions~\cite{Spector}.  To do this, they mollify the discontinuous
solution, and declare $y$ to be a slic (Singular Limiting
Induced from Continuum) solution, if for any mollifier
$\phi\in C^{\infty}(\RR)$, the mollified solution
$y^{\epsilon}=y*\phi_{\epsilon}$ satisfies
\[
  f^{\epsilon}:=  y^{\epsilon}_{tt}-\tau(y^{\epsilon}_x)_x \to 0 \, .
\]
Specifically, for the particular solutions found in \cite{Spector},
the authors discover the limiting value of the energetic cost of
opening up a cavity, and inclusion of this term implies that the
cavities are non-entropic.

Here we compare our methods and results to those of \cite{GT}.  We
begin with their one-dimensional example, which is a slic solution,
and show that it is a weak* solution.  We similarly calculate the
entropy production, and show that this agrees with the limit obtained
in \cite{GT}.  Our approach has several advantages: by
extending the stress and entropy via \cref{tauhat} and \cref{What},
respectively, we do not need to work directly with mollifiers;
our calculations are exact so there are no approximation errors;
because we are working in the space of measures, we do not need
to integrate explicitly; our approach is general and yields a
localized description of all waves; and our calculations themselves
are much shorter.

We begin with the solution studied in \cite{GT}: this is a
self-similar discontinuous motion $y(t,x)$ that represents a shearing
motion with the fracture that is initiated at time $t=0$ and then
propagates outwards, behind an expanding shock wave.  Explicitly $y$
is given by
\begin{equation}\label{ydef}
\begin{aligned}
   y(t,x) &= \lambda\, x\, \big(1-H(x+\sigma\, t)\big)\\
   &\qquad{}  +(\alpha\, x - Y_0\,t)\,\big(H(x+\sigma\,t)-H(x)\big)\\
   &\qquad{} +(\alpha\, x + Y_0\,t)\,\big(H(x)-H(x-\sigma\,t)\big)\\
   &\qquad{} + \lambda\, x\, H(x-\sigma t) .
\end{aligned}
\end{equation}
Here $Y_0$ is the velocity of the crack, $\sigma$ is the shock speed,
$\lambda$ is the initial stretching and $\alpha<\lambda$ is a free
parameter repressnting the strain at the cavity surface, $\a = u_*$ in
\cref{bdry}.  The parameters are related via the Rankine-Hugoniot conditions,
\begin{equation}\label{RHcond}
  Y_0 = \sigma (\lambda-\alpha), \qquad
  \sigma^2 (\lambda-\alpha)=\tau(\lambda)-\tau(\alpha)\,.
\end{equation}

As above, we work with the associated first order system \cref{elast}.
For the given motion, the components $U \in C^1\big(0,T;M(\RR)\big)$
and $v\in C^1\big(0,T;BV(\RR)\big)$ are easily computed to be
\begin{equation}\label{uvcomp}
\begin{aligned}
   U(t)&=y_x = 2\,t\, Y_0\, \delta_0 +
     \lambda \,(1-H_{-\sigma\, t}) + \alpha\,(H_{-\sigma\, t}-H_{\sigma\, t})
     + \lambda\, H_{\sigma\, t},  \\
   v(t)&=y_t = - Y_0\, (H_{-\sigma\, t} - H_0) + Y_0 \,( H_0 - H_{\sigma\, t}),
\end{aligned}
\end{equation}
where $H_a$ stands for the shifted Heaviside function $x \to H(x-a)$,
and  for convenience we have dropped the inclusion $\iota:L^1\to M$.
It is easy to check that the generalized Rankine-Hugoniot conditions
\cref{tauRH} reduce to \cref{RHcond} for this solution.

A direct computation reveals that
\[
  U_t = 2\, Y_0\, \delta_0 -
     \sigma \,(\lambda -\alpha)\,(\delta_{-\sigma\, t}
     +\delta_{\sigma\, t}) \com{and}
  v_x = 2\, Y_0\, \delta_0 -
    Y_0\,(\delta_{-\sigma\, t}+\delta_{\sigma\, t}),
\]
so that \cref{elast}$_1$ holds in $C^1(0,T; M(\RR))$.  While in
\cite{GT} the authors use mollifiers and slic solutions to deal with
the nonlinear term $\tau(y_x)$, we calculate directly using the
extension \cref{tauhat}.  Using \cref{tauhat} in \cref{uvcomp}, we get
\[
  \wh\tau(U) = L_\tau\,2\,t\, Y_0\, \delta_0 +
     \tau(\lambda) \,(1-H_{-\sigma\, t})
   + \tau(\alpha)\,(H_{-\sigma\, t}-H_{\sigma\, t})
     + \tau(\lambda)\, H_{\sigma\, t},
\]
so that
\[
  \wh\tau(U)_x = L_\tau\,2\,t\, Y_0\,\DD_x \delta_0 +
     \big(\tau(\lambda)-\tau(\alpha)\big)
   \,(\delta_{\sigma\, t}-\delta_{-\sigma\, t}),
\]
while from \cref{uvcomp},
\[
  v_t =  \sigma \, Y_0 \,(\delta_{\sigma\, t}  - \delta_{-\sigma\, t } )\,.
\]
Using \cref{RHcond}, \cref{elast}$_1$ holds in $C^1(0,T; M(\RR))$, so
we conclude that $y$ is a weak* solution, if and only if $L_\tau = 0$.
This is consistent with \cref{thm:elast} and with the
(necessary and sufficient) condition in \cite{GT} for slic solutions.
Thus, the approximation procedure of \cite{GT} applied to their crack
initiation example can be replaced by extending the stress
$\widehat{\tau}$ and working directly with singular measures.

We now compute the entropy and entropy production of the solution
\cref{uvcomp}.  Setting $L_\tau = 0$ and recalling the entropy is $\eta
= \frac12v^2 + W(u)$ with entropy flux $q = v\,\tau(u)$, we again
extend and write
\[
\begin{aligned}
  \widehat{q} &= -v \,\wh{\tau}(U)
     = \Big( Y_0 \big( H_{-\sigma\, t} - H_0 \big)
       - Y_0 \big( H_0 - H_{\sigma\, t} \big) \Big)\, \tau(\alpha)\,,
 \com{and} \\
  \widehat{\eta} & =\TS{\frac12}{Y_0^2}\, \big( H_{-\sigma t} -
  H_{\sigma t} \big) + L_W\, 2\, t\, Y_0\, \delta_0 \\
  & \qquad{} + W(\lambda)\,\big((1-H_{-\sigma\, t}) + H_{\sigma\, t}\big)
  +W(\alpha)\,\big(H_{-\sigma\, t}-H_{\sigma\, t}\big)\,.
\end{aligned}
\]
Differentiating, we get
\[
\begin{aligned}
  \del_x \wh{q} &= Y_0\, \tau(\alpha)\, \big(
     \delta_{-\sigma \,t} +\delta_{\sigma\, t}- 2\,\delta_0 \big), \com{and}\\
  \del_t \wh{\eta} &=
   \TS{\frac12}{Y_0^2}\,\sigma\,\big(\delta_{-\sigma t}+\delta_{\sigma
     t} \big)  + L_W\, 2\, Y_0\, \delta_0
   - \sigma\, \big(W(\lambda)-W(\alpha)\big)\,
    \big(\delta_{-\sigma t}+\delta_{\sigma t}\big),
\end{aligned}
\]
so that the entropy production is
\[
\begin{aligned}
  \del_t \widehat{\eta} + \del_x \widehat{q} &=
  \theta\, \big( \delta_{-\sigma\, t} + \delta_{\sigma\, t}\big)
   +2\, Y_0\, \big(\tau_{\infty}-\tau(\alpha)\big)\,\delta_0,
\com{where}\\
     \theta &:= \sigma \big( \TS{\frac12}Y_0^2 + W(\alpha)-W(\lambda)\big)+\tau(\alpha)\,Y_0\,,
\end{aligned}
\]
where we have used $L_W = \tau_\infty$.  Finally, using \cref{RHcond}
and manipulating, we get
\begin{equation}
\begin{aligned}
  \theta &= \sigma\,\Big(\TS{\frac12}
   \big(\tau(\lambda)+\tau(\alpha)\big)\,(\lambda-\alpha)
  -\big(W(\lambda)-W(\alpha)\big)\Big)\\
     & = \sigma \int_{\alpha}^{\lambda}
   \Big(\TS{\frac12}\big(\tau(\lambda)+\tau(\alpha)\big)
       -\tau(s)\Big)\; ds <0,
\end{aligned}
\label{theta}
\end{equation}
because $\tau''<0$ for a stress with softening response.

It is now clear that the shocks with speed $\pm\sigma$ have negative
entropy production, as needed, but the crack at the origin does not,
so is not entropic.  This again mirrors the results of \cite{GT} and
\cref{thm:elast}.

Finally, in \cite{GT}, the authors define the total mechanical energy
of the slic-solution $y$ on the interval $I \subset \RR$ via the limit
\[
  E_{slic}(y;I) := \lim_{\epsilon \to 0} \int_{I}
  \Big(\TS{\frac{1}{2}}(y_t^{\epsilon})^2+W(y_x^{\epsilon})\Big)\;dx.
\]
However, in our framework the total energy on the interval is simply
\[
  \wh\eta(I) = \int_I \Pi(\wh\eta) \;dx + \wh\eta_s(I),
\]
where as usual $\wh\eta_s$ is the singular part.  In particular, if
\[
  (-\sigma\,t,\sigma\,t)\subset I, \com{we have}
  \iota\big(H_{-\sigma\, t}-H_{\sigma\, t}\big)(I) = 2\,\sigma\,t,
\]
so we immediately obtain
\[
\begin{aligned}
  \wh\eta(I) &= W(\lambda)\,(b-a) + L_W\,2\,t\,Y_0
   + \big(\TS{\frac12}Y_0^2 +
   W(\alpha)-W(\lambda)\big)\,2\,\sigma\,t\\
  &= \wh\eta_{nc}(I) + 2\,t\,\big(
     \theta + Y_0\,(\tau_\infty - \tau(\alpha))\big),
\end{aligned}
\]
where $\wh\eta_{nc}$ is the entropy of the crack-free solution $y=\lambda\,x$.
Using \cref{theta} and \cref{RHcond}, we calculate
\[
  \theta + Y_0\,(\tau_\infty - \tau(\alpha))
  = \sigma\,\int \TS{\frac12}\big(\tau(\lambda)-\tau(\alpha)\big)
     + \big(\tau_\infty - \tau(s)\big)\;ds >0,
\]
so that the crack-free solution has lower energy, as noted in~\cite{GT}.
Once again we see the advantage of working directly with measures, and
avoiding explicit integrations.

\end{document}